\documentclass[reqno]{amsart}

\usepackage[colorinlistoftodos,textsize=small]{todonotes}

\usepackage{amstext}
\usepackage{amsmath}
\usepackage{amssymb}
\usepackage{latexsym}
\usepackage{amsthm}
\usepackage{graphicx}

\usepackage[utf8]{inputenc}
\usepackage[T1]{fontenc}

\usepackage[nobysame,alphabetic,initials]{amsrefs}
\usepackage{bbm}

\usepackage{orcidlink}

\allowdisplaybreaks

\usepackage{enumerate}

\usepackage{hyperref}
\usepackage{xcolor}

\usepackage[normalem]{ulem}

\newcommand{\R}{\mathbb{R}}
\newcommand{\Ex}{\mathbb{E}}
\newcommand{\N}{\mathbb{N}}
\newcommand{\Sp}{\mathbb{S}}
\newcommand{\Hy}{\mathbb{H}}

\newcommand{\dx}{{\rm d}}
\newcommand{\indi}{\mathbf{1}}
\renewcommand{\d}{\mathrm{d}}

\DeclareMathOperator{\inti}{int}
\DeclareMathOperator{\bd}{\partial}
\DeclareMathOperator{\diam}{diam}

\DeclareMathOperator{\Vol}{Vol}
\DeclareMathOperator{\Var}{Var}
\DeclareMathOperator{\Vis}{Vis}

\newcommand{\var}{{\rm Var}\,}

\newcommand{\aff}{{\rm aff}\,}
\newcommand{\F}{\mathcal{F}}
\newcommand{\K}{\mathcal{K}}
\newcommand{\M}{\mathcal{M}}
\newcommand{\PP}{\mathbb{P}}

\newtheorem{tetel}{Theorem}
\newtheorem{lemma}{Lemma}
\newtheorem{corollary}{Corollary}

\title[]{Variances and Central Limit Theorems for Random Beta-Polytopes and in Other Geometric Models}

\author[F.~Fodor]{Ferenc Fodor%$^{\orcidlink{0000-0001-9747-1981}}$
}
\address{Bolyai Institute,  University of Szeged, Aradi v\'ertan\'uk tere 1, H-6720 Szeged, Hungary}
\email{fodorf@math.u-szeged.hu}

\author[B.~Gr\"unfelder]{Bal\'azs Gr\"unfelder%$^{\orcidlink{0009-0008-2447-3378}}$
}
\address{Bolyai Institute, University of Szeged, Aradi v\'ertan\'uk tere 1, H-6720 Szeged, Hungary}
\email{grunfelder.balazs@szte.hu}



\begin{document}
	
	\begin{abstract}
		We prove asymptotic lower and upper bounds with matching orders of magnitude on the variances of the intrinsic volumes and the number of $k$-faces of $d$-dimensional random beta-polytopes. Using Stein's method, we establish central limit theorems for the intrinsic volumes. 
		We also prove asymptotic upper bounds on the variances of the volume and vertex number of spherical random polytopes in spherical convex bodies, and hyperbolic random polytopes in convex bodies in hyperbolic space. Moreover, we consider a circumscribed model on the sphere.  
	\end{abstract}
	
	\maketitle
	
	\section{Introduction and results}
	
	The study of random polytopes generated as the convex hull of random points is a highly diverse field of geometry.
	The number and distribution of the random points, the different concepts of convexity, and the various geometric functionals considered all represent distinct models. For a general overview, detailed history, and references of the extensive literature of this beautiful and rapidly developing topic, see, for example, the surveys by B\'ar\'any \cite{Ba08}, Reitzner \cite{R10}, Schneider \cite{Sch18}, Weil and Wieacker \cite{WW93}, and the book by Schneider and Weil \cite{SW08}.   
	
	In this paper, we focus on rotationally symmetric distributions, such as beta-type distributions in $\R^d$ and the uniform distributions on the unit sphere $\Sp^{d}$ and in hyperbolic space $\Hy^d$. 
	
	\subsection*{Random beta-polytopes}
	First, we consider the so-called beta-distributions in $\R^d$. For $\beta>-1$, we say that a random point in the interior of the $d$-dimensional unit ball $B^d$ is beta-distributed if its probability density function is 
	\begin{equation*}
		f_{d,\beta}(x)=c_{d,\beta}(1-\|x\|^2)^\beta, \quad \|x\|<1,\quad \text{ with }\,
		c_{d,\beta}=\frac{\Gamma(\frac d2+\beta+1)}{\pi^{\frac d2}\Gamma(\beta+1)},
	\end{equation*}
	where $\|\cdot\|$ denotes the Euclidean norm and $\Gamma$ is Euler's gamma function.
	
	In \cite{KTT19}, some special cases are highlighted: $\beta=0$ yields the uniform model in which the random points are uniformly distributed in the unit ball; the uniform distribution on the sphere $\Sp^{d-1}$ is the weak limit of the beta-distribution as $\beta\to-1$; and suitably adjusting the beta-distribution leads to the Gaussian distribution.
	
	%Now we are focusing our attention on the spherically symmetric distributions in the $d$-dimensional Euclidean space $\R^d$, and, within that, on the so-called beta distributions. 
	The asymptotics of the expectation of various geometric quantities, as the number $n$ of points tends to infinity, have been studied, for instance, by Affentranger \cite{A91}, Carnal \cite{C70}, Dwyer \cite{D91}, Eddy and Gale \cite{EG81}, and Hashorva \cite{H11}. For further references and history on random beta-polytopes, see, for instance, \cite{K23}, \cite{KP25}, \cite{KS25}, \cite{KTT19}, \cite{M71}, and the references therein.
	
	In this paper, we consider the following model. 
	Let $x_1,\ldots,x_n$ be i.i.d. random points chosen in $B^d$ according to the beta-distribution for a fixed $\beta>-1$. Denote the convex hull of a set of points $P\subset\R^d$ by $[P]$.
	The convex hull $K_n^\beta=[x_1,\ldots,x_n]$ is called a random beta-polytope.
	The exact expectations of the number of facets and the intrinsic volumes have been determined by Kabluchko, Temesvari, and Th{\"a}le in \cite{KTT19}, among other results.
	We recall that the intrinsic volumes $V_s(K)$, $s=0,\ldots,d$, of a convex body (compact, convex set with nonempty interior) $K$ are defined as the coefficients of the Steiner formula
	$$\Vol_d(K+\varrho B^d)=\sum_{s=0}^{d}\varrho^{d-s}\kappa_{d-s}V_s(K),$$
	where $K+\varrho B^d$ is the Minkowski sum of $K$ and the ball of radius $\varrho\geq0$, and $\kappa_d$ denotes the volume of the $d$-dimensional unit ball.
	In particular, $V_d(K)=\Vol_d(K)$ is the $d$-dimensional volume (Lebesgue measure), $V_0(K)=1$, $V_1$ is a constant times the mean width, and $V_{d-1}$ is half of the surface area.
	For the $s$-dimensional volume ($s\leq d$), we will use $\Vol_s$.
	
	To the best of our knowledge, asymptotic results for the variances of the intrinsic volumes and central limit theorems are not known for general $\beta>-1$ (cf. \cite{KS25}*{Remark 1.4}), only in the uniform case, i.e., $\beta=0$.
	
	When the i.i.d. random points are chosen according to the uniform distribution in $B^d$, K\"ufer \cite{Kuf94} proved upper bound for the variance of the volume of the random polytope.
	Reitzner \cite{R03},\cite{R05} proved lower and upper bounds with matching order of magnitude for the variance of the volume when the uniform random points are chosen in a convex body $K$ with $C^2_+$ smooth boundary.
	B\"or\"oczky, Fodor, Reitzner and V\'igh \cite{BFRV09} proved an upper bound on the variance of the mean width of the random polytope in the case when $K$ has a rolling ball,
	and for convex bodies with $C^2_+$ smooth boundary, B\'ar\'any, Fodor, and V\'igh \cite{BFV10} proved matching lower and upper bounds for the order of magnitude of the variance of the intrinsic volumes of the random polytope.
	Th\"ale, Turchi and Wespi \cite{TTW18} proved central limit theorems for the intrinsic volumes.
	For the number of $k$-dimensional faces $f_k$, ($k\in\{0,1,\ldots,d-1\}$), Reitzner \cite{R05}, \cite{R05b} proved asymptotic lower and upper bounds on the variance with matching orders of magnitude.
	
	%Furthermore, in the case when the random points are uniformly distributed on the boundary of a smooth convex body in $\R^d$, Reitzner \cite{R03}, Richardson, Vu and Wu \cite{RVW08} proved asymptotic upper and lower bounds on the variance of the volume of the random polytope, and a central limit theorem is proved by Th\"ale \cite{Th18}.
	%Turchi and Wespi \cite{TW18} proved upper and lower bounds on the variance, a strong law of large numbers, and a quantitative central limit theorem for the intrinsic volumes. 
	%Reitzner and Sonnleitner \cite{RS25} proved variance bounds and central limit theorems for the number of $k$-dimensional faces.\todo{Ez kell?}
	
	For real functions $f$ and $g$ that are defined on the same space $I$, we write  $f\ll g$  if there exists a constant $\gamma>0$ such that $|f(x)|\leq \gamma g(x)$ for all $x\in I$. If $f\ll g$ and $g\ll f$, we write $f\approx g$.
	
	One of our main results is the following theorem, which determines the order of magnitude of $\var V_s(K_n^\beta)$.
	\begin{tetel}\label{thm:beta-main}
		For $s=1,\ldots, d$,
		\begin{equation*}
			\var V_s(K_n^\beta)\approx n^{-\frac{d+3}{d+1+2\beta}}.
		\end{equation*}
	\end{tetel}
	
	The upper bound in Theorem~\ref{thm:beta-main} implies a strong law of large numbers, i.e., for $s=1,\ldots,d$, $V_s(K_n^\beta)$ almost surely tends to its expected value, as $n\to\infty$.
	
	%that can be proved by standard arguments (see, for instance, \cite{BFV10}, \cite{R03}). 
	% \begin{corollary}\label{beta-slln}
		% For $s=1,\ldots,d$, it holds with probability one, that
		%     $$\lim_{n\to\infty}V_s(K_n^\beta)=\Ex V_s(K_n^\beta).$$
		% \end{corollary}
	
	The lower bound on the variances can be used to prove central limit theorems for $V_s(K_n^\beta)$. 
	Let $d_W(\cdot,\cdot)$ denote the Wasserstein distance of two random variables, defined as
	$$d_W(X,Y):=\sup_{h\in \text{Lip}_1}|\Ex h(X)-\Ex h(Y)|,$$
	where the supremum is taken over all Lipschitz functions $h:\R\to\R$ with Lipschitz constant at most $1$.
	We use the fact that if $G$ is a standard Gaussian random variable and $(W_n)_{n\in\N}$ is a sequence of centered random variables with finite second moments for which $d_W(W_n/\sqrt{\var W_n},G)\to 0$, as $n\to\infty$, then $W_n/\sqrt{\var W_n}$ converges in distribution to $G$.
	
	Similarly to \cite{TTW18}, we use the normal approximation bound based on Stein's method \cite{Ste86}, which had been used by, for instance, Besau and Th\"ale \cite{BT20}, Chatterjee \cite{Cha08}, Fodor and Papv\'ari \cite{FP24}, and Lachi\`eze-Rey and Peccati \cite{LRP17}, for different models.
	
	In this paper, we prove the following central limit theorems on $V_s(K_n^\beta)$ for $s=1,\ldots,d$ and $\beta>-1$.
	\begin{tetel}\label{beta-CHT}
		Let $G$ be a standard Gaussian random variable. Then, for $s=1,\ldots,d$, depending on the value of $\beta$, the following hold.
		\begin{enumerate}[(i)]
			\item  For $-1<\beta\leq1$,
			\begin{equation*}
				d_W\Biggl(\frac{V_s(K_n^\beta)-\Ex V_s(K_n^\beta)}{\sqrt{\var V_s(K_n^\beta)}},G\Biggr)\ll n^{-\frac12+\frac{1+\beta}{d+1+2\beta}}(\log n)^{\frac{3(d+1)+2+2\beta}{d+1+2\beta}},            
			\end{equation*}
			\item and for $\beta>1$,
			\begin{equation*}
				d_W\Biggl(\frac{V_s(K_n^\beta)-\Ex V_s(K_n^\beta)}{\sqrt{\var V_s(K_n^\beta)}},G\Biggr)\ll n^{-\frac12+\frac{1+\beta}{d+1+2\beta}}(\log n)^{\frac{3(d+1)+1+3\beta}{d+1+2\beta}}.            
			\end{equation*}
		\end{enumerate}
	\end{tetel}
	
	Notice that the exponent of $n$ in both parts of Theorem~\ref{beta-CHT} is negative for any $\beta>-1$ and $d\geq2$, so the right-hand sides tend to zero as $n\to\infty$, and therefore, 
	\begin{equation*}
		\dfrac{V_s(K_n^\beta)-\Ex V_s(K_n^\beta)}{\sqrt{\var V_s (K_n^\beta)}} \xrightarrow{d} G,
	\end{equation*}
	where $\xrightarrow{d}$ means convergence in distribution.
	However, the rate of convergence depends on $\beta$, and the function in the exponent of the logarithm changes at $\beta=1$.
	
	We note that in the special case of $\beta=0$, Theorem~\ref{beta-CHT} gives back the results of \cite{TTW18} in the uniform model.
	The limit of the right-hand side of part $(i)$ of Theorem~\ref{beta-CHT}, as $\beta\to-1$, gives back the result of Th\"ale \cite{Th18}, where the random points are uniformly distributed on the boundary of a convex body with a smooth boundary.
	%For $\beta\geq(d-1)/4$, the upper bound in part $(ii)$ of Theorem~\ref{beta-CHT} tends to infinity as $n\to\infty$, thus, this method does not provide a convergence in distribution.
	
	We also prove matching asymptotic lower and upper bounds for the number of $k$-dimensional faces $f_k$, $k\in\{0,1,\ldots,d-1\}$, of the random beta polytope.
	For the expectation of the $f$-vector, exact formulas for finite $n$, as well as asymptotics as $n\to\infty$, have been proved by Kabluchko, Th\"ale and Zaporozhets \cite{KTZ20}.
	\begin{tetel}\label{thm:f-vector}
		For $k=0,1,\ldots,d-1$,
		$$\Var f_k(K_n^\beta)\approx n^{\frac{d-1}{d+1+2\beta}}.$$
	\end{tetel}

	The order of magnitude in Theorem~\ref{thm:f-vector} coincides with the uniform model for $\beta=0$, and for $d\geq4$, the limit as $\beta\to-1$ matches with \cite{RS25}.

	\subsection*{Non-Euclidean random polytopes}
	Random polytopes in non-Euclidean geometries have recently attracted much attention. It is a natural problem to try to transfer statements from Euclidean theory to more general settings. For an overview of such results, see, for example, Besau, Ludwig, and Werner \cite{BLW18}, Besau and Th\"ale \cite{BT20}, Kabluchko and Panzo \cite{KP25}, Schneider \cite{Sch22}, and the references therein.
	
	%There is an interesting relation between the non-uniform random models with weighted volumes in Euclidean space and the models in other geometries, such as spherical, hyperbolic spaces, or Hilbert geometries. 
	%In this section, we study this research direction and improve Theorem~\ref{main:upper-bound} with weaker assumptions.
	
	For $d\ge 2$, let $\M^d$ denote one of the following spaces: Euclidean $d$-space $\R^d$, the unit sphere $\Sp^d$ in $\R^{d+1}$, or the hyperbolic space $\Hy^d=\{x\in\R^{d+1}:x_1^2+\ldots x_d^2-x_{d+1}^2=-1,\; x_{d+1}>0\}$.
	Endow each space with the corresponding geodesic distance: the metric $d$ induced by the Euclidean scalar product $\langle\cdot,\cdot\rangle$ in $\R^d$, $d_S(x,y)=\arccos\langle x,y\rangle$ for $x,y\in\Sp^d$, and $\cosh d_H(x,y)=x_{d+1}y_{d+1}-x_1y_1-\ldots-x_dy_d$ for $x,y\in\Hy^d$.
	
	We call a set $K$ convex in $\M^d$ if, together with any two of its points, the unique geodesic segment connecting them is also contained in $K$. 
	A convex body is a compact convex set in $\M^d$ with non-empty interior. 
	We note that a spherical convex body is always contained in an open hemisphere.
	
	The volume (Lebesgue measure) of a measurable set in $\M^d$ is denoted by $\Vol_{\M^d}(\cdot)$, or specifically $\Vol_d(\cdot)$ in $\R^d$.
	Let $\mathcal{K}(\M^d)$ denote the family of convex bodies in $\M^d$. 
	%\textcolor{red}{We equip $\mathcal{K}$ with the corresponding Hausdorff metric.}
	The convex hull of a closed set $X\subset\M^d$ is the intersection of all closed half-spaces containing $X$ if $\M^d=\R^d$ or $\Hy^d$, and the intersection of all closed hemispheres containing $X$ if $\M^d=\Sp^d$.
	
	In this paper, we consider the following probabilistic model for random polytopes.
	Let $K\in\mathcal{K}(\M^d)$, and let $x_1,\ldots,x_n$ be $n$ independent random points chosen according to the uniform distribution in $K$.
	The convex hull of the points $x_1,\ldots,x_n$ is a random polytope $K_n$  in $K$.
	
	Let $H_{d-1}^{\M^d}(x)$ denote the general Gauss--Kronecker curvature at a boundary point $x\in\bd K$. The boundary of a convex body may not be twice differentiable, so its classical Gauss--Kronecker curvature may not exist at every point. A generalized notion of second-order differentiability can be introduced such that the boundary is differentiable in this sense at almost every point with respect to the surface area measure (cf. Alexandrov's theorem). Then, a generalized Gauss--Kronecker curvature can be defined at these points that coincides with the classical notion everywhere where the boundary is twice differentiable. The symbol $H_{d-1}^{\M^d}(x)$ refers to this generalized notion of Gauss--Kronecker curvature. For more details, see Sch\"utt and Werner \cite{SW23}*{Section~2}.
	%For $\Sp^d$ and $\Hy^d$, on the expected volume difference of $K$ and $K_n$, and the number of vertices $f_0$, 
	The spherical and hyperbolic cases of the following theorem were proved by Besau, Ludwig, and Werner \cite{BLW18}*{Theorems 2.2, 3.2, and Corollaries 2.3, 3.3}. The Euclidean case is due to Sch\"utt \cite{S94} extending results of B\'ar\'any \cite{B92}.
	%The Euclidean volume is denoted by $V(\cdot)$ for short, and $\alpha_d=V(B^d)$.
	
	\begin{tetel}%\label{thm:exp}
		Let $K\in\K(\M^d)$. If $K_n$ is the convex hull of $n$ random points chosen uniformly in $K$, then
		\begin{align*}
			&\lim_{n\to\infty} \Ex(\Vol_{\M^d}(K\setminus K_n))\cdot n^{\frac{2}{d+1}} = \gamma_d \Vol_{\M^d}(K)^{\frac{2}{d+1}} \int_{\bd K} H_{d-1}^{\M^d}(K,x)^{\frac1{d+1}}\dx x,\\
			&\lim_{n\to\infty} \Ex(f_0(K_n))\cdot n^{-\frac{d-1}{d+1}} = \gamma_d \Vol_{\M^d}(K)^{-\frac{d-1}{d+1}} \int_{\bd K} H_{d-1}^{\M^d}(K,x)^{\frac1{d+1}}\dx x. 
		\end{align*}
		where integration on $\partial K$ is with respect to the $(d-1)$-dimensional Hausdorff measure, and  $\gamma_d$ is an explicitly known constant that depends only on $d$.
		%\begin{equation}\label{beta}
		%    \beta_d=\frac{(d^2+d+2)(d^2+1)}{2(d+3)\cdot(d+1)!}\Gamma\left(\frac{d^2+1}{d+1}\right)\left(\frac{d+1}{\alpha_{d-1}}\right)^{\frac{2}{d+1}}.
		%\end{equation}
	\end{tetel}

	We say that a ball of radius $r>0$ rolls freely in $K$ ($K$ slides freely in a ball of radius $R>0$, resp.) if for any $x\in\bd K$ there exists a ball of radius $r$ ($R$, resp.) containing $x$ on its boundary and contained in $K$ (containing $K$, resp.).
	If $K$ has a rolling ball and slides freely in a ball at the same time, then $\bd K$ is $C^1$ and strictly convex. 
	However, $\bd K$ is not necessarily $C^2$. A simple example can be constructed in the plane by joining two circular arcs of different radii that share a tangent at their common endpoints, and the other endpoints are connected by a suitable $C^2_+$ arc. 
	One of our main results is the following theorem.
	
	\begin{tetel}\label{upper-bound2}
		Let $K\in \mathcal{K}(\M^d)$, and assume that $K$ has a rolling ball and slides freely in a ball. %Let $K_n$ be the convex hull of $n$ i.i.d. random points in $K$ chosen according to the uniform probability distribution. 
		Then
		\begin{align*}
			\Var \Vol_{\M^d} (K_n) &\ll n^{-\frac{d+3}{d+1}},\\
			\Var f_0 (K_n) &\ll  n^{\frac{d-1}{d+1}}.
		\end{align*}
	\end{tetel}
	The variance upper bound on the volume implies a strong law of large numbers that can be proved by standard arguments (see, for instance, \cite{BFV10}, \cite{R03}). 
	\begin{corollary}%\label{SLLN2}
		Under the same assumptions as in Theorem~\ref{upper-bound2}, it holds with probability one that
		\begin{equation*}
			\lim_{n\to\infty} (\Vol_{\M^d}(K\setminus K_n))\cdot n^{\frac{2}{d+1}} = \gamma_d \Vol_{\M^d}(K)^{\frac{2}{d+1}} \int_{\bd K} H_{d-1}^{\M^d}(K,x)^{\frac1{d+1}}\dx x.
		\end{equation*}
	\end{corollary}
	
	The number of vertices is not a monotone function of $n$, but using the fact that it may be increased by at most one when another point is added, for $d\ge4$, similarly to \cite{R03}, a strong law of large numbers can be proved for $f_0$ as well.
	In $\R^2$, B\'ar\'any and Steiger \cite{BS13} proved variance upper bounds and strong laws of large numbers for the area and the number of vertices of a random polygon with no smoothness condition on $K$. 
	However, the method they used for the upper bounds can not be extended to higher dimensions.
	\begin{corollary}
		%\label{SLLN2-f0}
		Under the same assumptions as in Theorem~\ref{upper-bound2}, for $d\ge4$, it holds with probability one that
		\begin{equation*}
			\lim_{n\to\infty} f_0(K_n)\cdot n^{-\frac{d-1}{d+1}} = \gamma_d \Vol_{\M^d}(K)^{-\frac{d-1}{d+1}} \int_{\bd K} H_{d-1}^{\M^d}(K,x)^{\frac1{d+1}}\dx x.
		\end{equation*}
	\end{corollary}
	
	%A direct proof of Theorem~\ref{upper-bound2} with slightly stronger smoothness condition (assuming that $\bd K$ is $C^2_+$ smooth) can be given by applying Theorem~\ref{gazdasagos-sapkafedes}. 
	
	%The lower bound \eqref{s-lower} holds assuming the same smoothness conditions. We prove matching upper bounds.
	
	Since Besau and Th{\"a}le \cite{BT20} proved in $\Sp^d$ and $\Hy^d$ that if $K$ 
	has $C_+^2$ boundary, then $\Var \Vol_{\M^d} (K_n) \gg n^{-\frac{d+3}{d+1}}$, Theorem~\ref{upper-bound2} yields the following.
	\begin{tetel}\label{main:upper-bound}
		Let $K\in\mathcal{K}(\M^d)$ with $C_+^2$ boundary. %Let $K_n$ be the convex hull of $n$ i.i.d. random points in $K$ chosen according to the uniform probability distribution. 
		Then
		\begin{equation*}
			\Var \Vol_{\M^d} (K_n) \approx n^{-\frac{d+3}{d+1}}.
		\end{equation*}
	\end{tetel}
	For $\M^d=\R^d$, both the lower and the upper bounds are due to Reitzner \cites{R03,R05}.
	
	Our proof of Theorem~\ref{upper-bound2} uses results of a weighted model described in Subsection~\ref{weighted}, and thus it is indirect. %We also give a direct proof of Theorem~\ref{upper-bound2} in Section~\ref{sec:4} via a non-euclidean version of the economical cap theorem that avoids using spherical integral geometry.

	\subsection*{Circumscribed model}
	The probability model discussed in this section considers circumscribed spherical polytopes containing a convex body $K$. 
	This model is naturally connected to the inscribed ones discussed via spherical polarity.
	The role of random points inside of $K$ is replaced by random closed hemispheres containing $K$, and the intersection of $n$ such hemispheres is a random polytope containing $K$.
	This model was studied, for example, by Besau, Ludwig and Werner \cite{BLW18}.

	We note that the hyperbolic model is not considered here, as 
	we are not aware of the expectations of the corresponding random variables.
	%polarity in $\Hy^d$ does not yield such a connection between inscribed and circumscribed random polytopes analogously to the spherical polarity, and even
	%it seems that the expectations of the corresponding random variables are not known.
	
	%Let $\mu$ be the unique rotation invariant probability measure on $\Sp^d$.% $\mu$ on the space $\mathcal{H}$ of closed hemispheres.
	Let $\mathcal{H}$ denote the space of closed hemispheres in $\Sp^d$.
	Each point $x\in\Sp^d$ is the pole of a unique closed hemisphere $H^-(x)=\{y\in\Sp^d:\langle x,y\rangle\leq0\}$.
	Thus, $\mathcal{H}$ inherits a natural topology from $\Sp^d$.
	We define the measure of a Borel set $A\subset\mathcal{H}$ as
	$$\mu(A)=\frac{1}{\omega_{d+1}}\int_{\Sp^d}\indi(H^-(x)\in A)\, \dx x,$$
	where $\omega_{d+1}$ is the surface volume of $\Sp^d$ and integration is with respect to spherical Lebesgue measure, and $\indi(\cdot)$ denotes the indicator function of an event.
	
	Let $\mathcal{H}_K=\{H^-\in\mathcal{H}:K\subset H^-\}$. 
	Choose $n$ i.i.d. random hemispheres containing $K$ according to the uniform distribution with the probability measure $\mu_K=\mu/\mu(\mathcal{H}_K)$.
	The intersection of these hemispheres is a random polytope containing $K$, denoted by $K^{(n)}$.
	
	The spherical polar $K^\ast$ of a convex body $K$ is defined as
	$K^\ast=\bigcap_{x\in K}H^-(x).$
	Since polarity reverses set inclusion, a hemisphere $H^-(x)$ contains $K$ if and only if $x\in K^\ast$.
	The polar body $(K^{(n)})^\ast$ is a polytope contained in $K^\ast$, and it is the convex hull of those $n$ i.i.d. random points in $K^\ast$ that are the poles of the random hemispheres containing $K$.
	This provides a direct connection between the inscribed and circumscribed models.
	
	Based on this connection, Besau, Ludwig, and Werner \cite{BLW18} proved asymptotic formulas for the expectation of the spherical mean width and the number of facets $f_{d-1}$ of $K^{(n)}$.
	The spherical mean width $U_1(K)$ is defined as
	$$U_1(K)=\frac{1}{2}\int_{G(d+1,d)}\chi(K\cap H)\,\d \nu (H),$$
	where $G(d+1,d)$ denotes the Grassmannian of the $d$-dimensional linear subspaces of $\R^{d+1}$, $\nu$ is the unique rotation invariant probability measure on $G(d+1,d)$, and $\chi$ denotes the Euler characteristic. 
	
	\begin{tetel}[\cite{BLW18}, Corollary 2.6]
		Let $K\in\K(\Sp^d)$. If $K^{(n)}$ is the intersection of $n$ random hemispheres containing $K$ and chosen uniformly according to $\mu_K$, then
		\begin{align*}
			&\lim_{n\to\infty}\Ex_{\mu_K}(U_1(K^{(n)})-U_1(K))\cdot n^{\frac{2}{n+1}}=\frac{\gamma_d}{\omega_{d+1}}\Vol_{\Sp^d}(K^\ast)^{\frac{2}{d+1}}\int_{\bd K} H_{d-1}^{\Sp^d}(K,x)^{\frac{d}{d+1}}\dx x,\\
			&\lim_{n\to\infty}\Ex_{\mu_K}f_{d-1}(K^{(n)})\cdot n^{-\frac{d-1}{n+1}}=\gamma_d\Vol_{\Sp^d}(K^\ast)^{-\frac{d-1}{d+1}}\int_{\bd K} H_{d-1}^{\Sp^d}(K,x)^{\frac{d}{d+1}}\dx x.
		\end{align*}
		%where $\beta_d$ is the constant in \eqref{beta}.
	\end{tetel}
	
	Our variance upper bounds in Theorem~\ref{upper-bound2} yield the following upper bounds on the variances of $U_1(K^{(n)})$ and $f_{d-1}(K^{(n)})$ due to spherical polarity.
	\begin{corollary}%\label{circ-upper}
		Let $K\in\mathcal{K}(\Sp^d)$ that has a rolling ball and which slides freely in a ball. 
		Then
		%Let $K^{(n)}$ be the intersection of $n$ i.i.d. random closed hemispheres containing $K$ chosen uniformly according to $\mu_K$. Then
		\begin{align*}
			\Var U_1 (K^{(n)}) &\ll n^{-\frac{d+3}{d+1}},\\
			\Var f_{d-1} (K^{(n)}) &\ll  n^{\frac{d-1}{d+1}}.
		\end{align*}
	\end{corollary}

	\section{Geometric tools}%\label{sect:tools}
	
	\subsection{The economic cap covering theorem}%\label{subsect:ecct}
	Let $K\in\mathcal{K}(\mathcal{M}^d)$ and let $H^-$ be a closed half-space  if $\mathcal{M}^d=\R^d$ or $\Hy^d$, and a closed hemisphere if $\mathcal{M}^d=\Sp^d$. Let $H^+$ denote the other closed half-space (hemisphere) determined by $H^-$.
	For $v>0$, we define the convex floating body of $K$ as
	\begin{equation*}
		K[v]=\bigcap \{H^-\colon \Vol_{\mathcal{M}^d}(K\cap H^+)\leq v\}.
	\end{equation*}
	%and the wet part analogously to \cite{BL88} in the Euclidean case. \todo{lasd meg Besau, et al}
	%Let $K\in\K(\Sp^d)$ and $t>0$. 
	%Let us define the function $v:K\to\R$ as\todo{Ez itt nem jo, komplementer kellene...}
	%\begin{equation*}
	%v(x):=\min\{\Vol_{\Sp^d}(K\cap\Sp^+)\;|\;x\in\Sp^+,\;\Sp^+\text{ is a closed hemisphere}\}.
	%\end{equation*}
	%We call the set $K(t)=K(v\leq t)=\{x\in K\;|\;v(x)\leq t\}$ the wet part, and the set $K(v\geq t)=\{x\in K\;|\;v(x)\geq t\}$ the convex floating body of $K$ with parameter $t$.
	
	Euclidean convex floating bodies were defined by B\'ar\'any and Larman \cite{BL88}, and independently by Sch\"utt and Werner \cite{SW04}. 
	The spherical convex floating body was introduced by Besau and Werner \cite{BW16}, and hyperbolic convex floating bodies were defined by Besau and Werner \cite{BW18}. 
	
	The closure of the complement of $K[v]$ in $K$ is called the wet part of $K$ with parameter $v$, and it is denoted by $K(v)$.
	% It was proved in \cite{BW16}*{Theorem~2.1} that for any spherical convex body $K$
	% \begin{equation*}\label{wet-part}
		%     \Vol_{\Sp^d}(K(v))\approx v^{\frac{2}{d+1}}, \text{ as } v\to 0^+.
		% \end{equation*}
	
	In Sections~\ref{sec:beta-upper}, \ref{sec:CHT}, and \ref{sec:direct}, we are going to use the economic cap covering theorem in $\R^d$, that was proved by B{\'a}r{\'a}ny and Larman \cite{BL88} and  B\'ar\'any \cite{Bar89}. %They showed that the wet part of a convex body $K$ with parameter $t$ can be covered by caps that have a volume of order $t$ economically, i.e., not too many caps are necessary.
	
	\begin{tetel}[\cite{BL88}, \cite{Bar89}]\label{R-ecct}
		Assume that $K\subset \R^d$ is a convex body with $\Vol_d(K)=1$ and $0<v <v_0=(2d)^{-2d}$. Then there exist caps $C_1, \ldots, C_m$ and pairwise disjoint convex sets $C_1',\ldots, C_m'$ such that $C_i'\subset C_i$ for each $i$, and
		\begin{enumerate}[(i)]
			\item $\bigcup_1^m C_i'\subset K(v) \subset \bigcup_1^m
			C_i$,
			\item $\Vol_d(C_i') \gg v$ and $\Vol_d(C_i)\ll v$ for each $i$,
			\item for each cap $C$ with $C\cap
			K[v]=\emptyset$ there is a $C_i$ containing C.
		\end{enumerate}
	\end{tetel}

		\subsection{Weighted non-uniform models in $\R^d$}\label{weighted}
		%Consider the following probability model in the $d$-dimensional Euclidean space $\R^d$.  
		Let $K\in\mathcal{K}(\R^d)$ and let $\varrho: K\to (0,\infty)$ be a probability density function that is continuous in a neighbourhood of the boundary of $K$ (in $K$).
		Then $\PP_\varrho(A):=\int_A \varrho(x)\,\d x$ for any measurable set $A\subset K$. We denote the expectation and variance with respect to $\PP_\varrho$ by $\Ex_\varrho$ and $\var_\varrho$, respectively.
		Let $x_1,\ldots,x_n$ be i.i.d. random points from $K$ distributed according to $\PP_\varrho$, and let $K_{(n)}$ be their convex hull, a random polytope in $K$.
		
		Furthermore, let $\lambda:K\to(0,\infty)$ be a weight function that is integrable in $K$ and continuous in a neighborhood of $\bd K$ and let $V_\lambda(A)=\int_A \lambda(x)\,\d x$ for all measurable subsets $A\subset K$
		.
		For $\varrho\equiv1/\Vol_d(K)$ and $\lambda\equiv1$, we obtain the uniform model.
		
		In this model, with no smoothness condition on $K$, B{\"o}r{\"o}czky, Fodor and Hug \cite{BFH10} studied the asymptotic behavior of the expectation of the weighted volume of missed part of $K$ and the number of vertices $f_0(K_{(n)})$ of the random polytope $K_{(n)}$, and proved that
		\begin{align*}%\label{R-weighted-exp}
			\lim_{n\to\infty} \Ex_{\varrho} (V_\lambda(K_{(n)})) \cdot n^{\frac{2}{d+1}} &= \gamma_d \int_{\bd K} \varrho(x)^{\frac{-2}{d+1}} \lambda(x)H_{d-1}^{\R^d}(K,x)^{\frac1{d+1}}\d x,\\
			\lim_{n\to\infty} \Ex_{\varrho} (f_0(K_{(n)})) \cdot n^{-\frac{d-1}{d+1}}  &= \gamma_d \int_{\bd K} \varrho(x)^{\frac{d-1}{d+1}} H_{d-1}^{\R^d}(K,x)^{\frac1{d+1}}\d x. \notag
		\end{align*}
		%where $\beta_d$ is the constant in \eqref{beta}, $\kappa(x)$ denotes the generalized Gauss–Kronecker curvature at $x\in\bd K$ and 
		%where the integration is with respect to the $(d - 1)$-dimensional Hausdorff measure $\mathcal{H}^{d-1}$ on $\bd K$.
		
		Assuming that $\bd K$ is $C^2_+$, Besau and Th{\"a}le \cite{BT20} proved the following asymptotic variance lower bound and a central limit theorem: 
		\begin{equation}\label{R-weighted-lower}
			\var_\varrho(V_\lambda(K_{(n)}))\gg n^{-\frac{d+3}{d+1}},\qquad 
			\dfrac{V_\lambda(K_{(n)})-\Ex_\varrho V_\lambda(K_{(n)})}{\sqrt{\Var_\varrho V_\lambda (K_{(n)})}} \xrightarrow{d} G,
		\end{equation}
		%\begin{tetel}[\cite{BT20}, Theorem~4.1]\label{thm:R-weighted-lower}
		%   Let $K\subset \R^d$ a convex body with $C_+^2$ smooth boundary. 
		%   \begin{equation*}
			%       \var_\varrho(V_\lambda(K_{(n)}))\gg n^{-\frac{d+3}{d+1}},
			%   \end{equation*}
		%   where the implied constants depend only on $K,\varrho,\lambda$ and the dimension $d$.
		%\end{tetel}
		%\begin{tetel}[\cite{BT20}, Theorem~2.1]\label{R:CLT}
		%    With the same assumptions as in Theorem~\ref{thm:R-weighted-lower},
		% \begin{equation*}
			%      \dfrac{V_\lambda(K_{(n)})-\Ex_\varrho V_\lambda(K_{(n)})}{\sqrt{\Var_\varrho V_\lambda (K_{(n)})}} \xrightarrow{d} G,
			%  \end{equation*}
		as $n\to\infty$, where $G$ is a standard normal random variable.
		%\end{tetel}
		The lower bounds in \cite{BT20} for $\Sp^d$ and $\Hy^d$ were deduced from \eqref{R-weighted-lower} by choosing particular weight functions that come from the gnomonic projections.
		
		Under weaker smoothness assumptions, Bak{\'o}-Szab{\'o} and Fodor \cite{BSF24} proved matching asymptotic upper bounds for the variance of the weighted volume and the number of vertices.
		
		%We say that a ball of radius $r>0$ rolls freely in $K$ if for any $x\in\bd K$ there exists a $v\in \R^d$ such that $x\in rB^d+v\subset K$. 
		%We say that $K$ slides freely in a ball of radius $R>0$ if for each $y\in R S^{d-1}$ there exists $p\in\R^d$ with $y\in K+p\subset RB^d$. 
		%If $K$ has a rolling ball and slides freely in a ball at the same
		%time, then $\bd K$ is $C^1$ and strictly convex. 
		%However, $\bd K$ need not be $C^2$.
		
		\begin{tetel}[\cite{BSF24}, Theorem~1.1]\label{thm:R-weighted-upper}
			For a convex body $K\subset\R^d$ that has a rolling ball and which slides freely in a ball, it holds that
			\begin{align*}
				\var_\varrho(V_\lambda(K_{(n)}))&\ll n^{-\frac{d+3}{d+1}},\\
				\var_\varrho(f_0(K_{(n)}))&\ll n^{\frac{d-1}{d+1}},
			\end{align*}
			where the implied constants depend only on $K,\varrho,\lambda$ and the dimension $d$.
		\end{tetel}
		We use Theorem~\ref{thm:R-weighted-upper} in our proof of Theorem~\ref{upper-bound2}.
		
		\subsection{Gnomonic projection} 
		
		%One of the projections that maps spherical convex sets into convex sets in the Euclidean space is 
		The gnomonic projection maps a $d$-dimensional open hemisphere from the origin radially to a tangent hyperplane. We refer to the point of tangency as the center of the projection.
		We may assume that the convex body $K\subset\Sp^d$ is contained in the (upper) open hemisphere $\Sp^d_+=\{x\in\Sp^d:\langle x, e_{d+1}\rangle>0\}$, and the center of projection is $e_{d+1}$.
		Then the gnomonic projection $g:\Sp^d_+\to\R^d$ is
		$$g(x)=\frac{x}{\langle x, e_{d+1}\rangle}-e_{d+1},$$
		and the hyperplane $\{x\in\R^{d+1}:\langle x, e_{d+1}\rangle=0\}$ is identified with $\R^d$.
		The map $g$ is bijective and $C^\infty$.
		Geodesic arcs of $\Sp^d_+$ are mapped into straight line segments in $\R^d$, thus the image of a spherical convex body is a convex body in the Euclidean sense.
		The gnomonic image of a set $X$ is often denoted by $\overline{X}=g(X)$.

		The gnomonic projection can be defined in the hyperbolic model similarly: $h:\Hy^d\to\R^d$ maps the points of the hyperboloid centrally to the tangent hyperplane $H=\{x\in\R^{d+1}:\langle x, e_{d+1}\rangle=1\}$, then identify $H$ with $\R^d$, i.e. identify
		$$h(x)=\left(\frac{x_1}{x_{d+1}},\dots,\frac{x_d}{x_{d+1}}\right).$$
		As $\|h(x)\|<1$ for all $x\in\Hy^d$, the image $g(\Hy^d)$ is the open unit ball of $\R^d$.
		The map $h$ is $C^\infty$, geodesic arcs are mapped into straight line segments, and it is bijective between $\Hy^d$ and $\inti B^d$.
		The image of a hyperbolic convex body is a convex body in $\R^d$. 
		%\begin{lemma}\label{gnom-masodrendu}
		%    Let $D$ be a spherical ball contained in the upper open hemisphere (or a hyperbolic ball in $\Hy^d$).
		%    Then $g(D)$ (or $h(D)$, resp.) is an ellipsoid.
		%\end{lemma}
		%\begin{proof}
		%We only show the spherical part, the hyperbolic case is similar.
		%    Let us change our point of view and consider the real projective space $\R\PP^d$. 
		%    The upper open hemisphere of $\Sp^d$, together with the equator whose antipodal points are identified, is a model of $\R\PP^d$.
		%    The tangent hyperplane $H=\{x\in\R^{d+1}:\langle x, e_{d+1}\rangle=1 \}$ extended by the ideal points is another model of $\R\PP^d$.
		%    The gnomonic projection $g:\Sp^d_+\to H$ preserves lines, thus it is a collineation in $\R\PP^d$.
		
		%    Any collineation maps second-order surfaces into second-order surfaces. If a spherical ball does not intersect the equator, i.e., all of its points are ordinary, then its image under a collineation will also have only ordinary points. 
		%    Such a second-order surface must be an ellipsoid.
		%\end{proof}

		%\subsection{Spherical and hyperbolic convex floating bodies}

		\section{Proof of the upper bound in Theorem~\ref{thm:beta-main}}\label{sec:beta-upper}
		
		For a measurable set $A\subset\R^d$, the probability that a beta-distributed random point falls in $A$ is called the probability content (or beta-content) of $A$, denoted by $\PP_\beta(A)$.
		A crucial part of the proof is to determine the order of magnitude of the probability content of a cap $C$ of height $t$.
		
		\begin{lemma}\label{sapka-PC}
			Let $\varepsilon>0$ and $0<t\leq 1-\varepsilon$. Then
			\begin{equation*}
				\PP_\beta(C(t))\approx t^{\frac{d+1+2\beta}{2}}, \,\text{ as }t\to 0^+.
			\end{equation*}
		\end{lemma}
		
		\begin{proof}
			We define the following functions: the one-dimensional cumulative distribution function $F_{1,\beta}$, the beta function $B(a,b)$ and the incomplete beta function $B(z;a,b)$.
			\begin{align*}
				F_{1,\beta}(h)&=c_{1,\beta} \int_{-1}^h (1-x^2)^\beta\,\d x, \quad h\in[-1,1],\\
				B(a,b)&= \int_0^1 x^{a-1}(1-x)^{b-1}\,\d x, \quad a,b>0,\\
				B(z;a,b)&= \int_0^z x^{a-1}(1-x)^{b-1}\,\d x, \quad 0\leq z\leq1, \quad a,b>0.   
			\end{align*}
			Then, after the substitution $y=1-x^2$, the following hold:
			\begin{align*}
				F_{1,\beta}&=1-c_{1,\beta}\int_h^1 (1-x^2
				)^\beta\,\d x
				%&=1-\frac{c_{1,\beta}}2\int_{h^2}^1 (1-u)^\beta u^{-1/2}\,\d u \\
				=1-\frac{c_{1,\beta}}2\int_0^{1-h^2} y^\beta (1-y)^{-1/2}\,\d y \\
				&=1-\frac{c_{1,\beta}}2 B(1-h^2;\beta+1,\frac12).
			\end{align*}

			According to \cite{KTT19}*{Lemma 4.5.}, 
			\begin{equation*}
				\PP_\beta(C(t))=1-F_{1,\beta+\frac{d-1}{2}}(1-t)
				=\frac{c_{1,\beta+\frac{d-1}2}}{2}B\Big(1-(1-t)^2;\beta+\frac{d+1}2,\frac12\Big).
			\end{equation*}
			
			We apply the following series expansion from \cite{AS64}*{Equation 26.5.4} on incomplete beta functions with $0<z<1$.
			\begin{equation*}
				B(z;a,b)=\frac{z^a(1-z)^b}a\Bigg(1+\sum_{n=0}^\infty \frac{B(a+1,n+1)}{B(a+b,n+1)}z^{n+1} \Bigg).
			\end{equation*}
			Therefore, 
			\begin{multline*}
				\PP_\beta(C(t))=\frac{c_{1,\beta+\frac{d-1}2}}{d+1+2\beta}t^{\frac{d+1+2\beta}2}(2-t)^{\frac{d+1+2\beta}2}(1-t)\\
				\times\Bigg(1+\sum_{n=0}^\infty \frac{B(\beta+\frac d2+\frac32,n+1)}{B(\beta+\frac d2+1,n+1)}\big(1-(1-t)^2\big)^{n+1} \Bigg).
			\end{multline*}
			
			Notice that for fixed $b$, the beta function $B(a,b)$ is monotonically decreasing in $a$.
			Hence,
			\begin{equation*}
				\sum_{n=0}^\infty \frac{B(\beta+\frac d2+\frac32,n+1)}{B(\beta+\frac d2+1,n+1)}\big(1-(1-t)^2\big)^{n+1}\leq \sum_{n=0}^\infty \big(1-(1-t)^2\big)^{n+1}=\frac{1-(1-t)^2}{(1-t)^2},
			\end{equation*}
			where the right-hand side is bounded from above by $\frac1{\varepsilon^2}-1$ since $t\leq1-\varepsilon$.
			
			Thus, we obtain the following inequalities on $\PP_\beta(C(t))$.
			\begin{equation*}
				\frac{c_{1,\beta+\frac{d-1}2}}{d+1+2\beta}(1+\varepsilon)^{\frac{d+1+2\beta}2} \varepsilon \cdot t^{\frac{d+1+2\beta}2}
				\leq\PP_\beta(C(t))\leq
				\frac{c_{1,\beta+\frac{d-1}2}}{d+1+2\beta}2^{\frac{d+1+2\beta}2}\frac1{\varepsilon^2}\cdot t^{\frac{d+1+2\beta}2}.
			\end{equation*}
		\end{proof}

		% \textcolor{orange}{
			% The one-dimensional cumulative distribution function is denoted by
			% \begin{equation*}
				%     F_{1,\beta}(h)=c_{1,\beta} \int_{-1}^h (1-x^2)^\beta\,\d x, \quad h\in[-1,1].
				% \end{equation*}
			% }
		
		%\textcolor{orange}{
			%According to \cite{KTT19}*{Lemma 4.5., p. 95.}, the probability content of cap of height $t$ is
			%$$\PP_\beta(C(t))=1-F_{1,\beta+\frac{d-1}{2}}(1-t),$$
			%which (after taking the power series of the integral) yields
			%\begin{equation}\label{sapka-PC}
			%    \PP_\beta(C(t))\approx t^{\frac{d+1+2\beta}{2}},
			%\end{equation}
			%as $t\to 0^+$.
			%}
		
		Later in this paper, we are going to apply the economic cap covering theorem, which requires its parameter $v$ to be at most $\kappa_d(2d)^{-2d}$.
		Thus, the height $t$ of the caps considered will be bounded by a constant depending on $d$ (and decreasing as $d$ increases).
		We may set $\varepsilon$ to a fixed value; for example, $\varepsilon=1/2$ will be suitable, as we only work with small $t$.
		%Then let $\xi_1=\xi_1(d,\beta)$ and $\xi_2=\xi_2(d,\beta)$ denote the coefficients of $t^{\frac{d+1+2\beta}2}$ in the lower and upper bounds, respectively, in Lemma~\ref{sapka-PC}.
		
		For the volume of a spherical cap $C(t)$ of height $t$, we have $\Vol_d (C(t))\approx t^{\frac{d+1}{2}}$.
		Hence, the volume of a cap can be converted to beta-content as follows.
		There exist positive constants $\gamma_1$ and $\gamma_2$, such that if $\Vol_d(C(t))=v$, then
		$$\gamma_1 v^{\frac{d+1+2\beta}{d+1}}\leq \PP_\beta(C(t))\leq \gamma_2v^{\frac{d+1+2\beta}{d+1}}.$$

		We recall a technical lemma from \cite{BFV10}.
		For given $z \in \Sp^{d-1}$ and $A \in G(d,s)$, their angle $\angle(z,A)$ is defined as the minimal angle $\angle(z,x)$ over all $x \in A$.
		\begin{lemma}[\cite{BFV10} Lemma 1, p. 609.]\label{szoges}
			Let $z \in \Sp^{d-1}$ and (small) $\alpha>0$ be fixed, then $\nu_s\{A\in G(d,s)\, | \, \angle(z,A)\leq
			\alpha\}\approx \alpha^{d-s}$.
		\end{lemma}
		
		Let $T_n$ denote the event that the floating body $\widetilde B=B^d[ (c\log n/n)^{\frac{d+1}{d+1+2\beta}} ]$, which is a concentric ball, is contained in $K_n^\beta$ for a suitable constant $c$ that will be specified later.
		There is a constant $\delta$ such that the probability of $T_n^c$, i.e., the complement of $T_n$, is at most $n^{-\delta c}$.
		This can be seen as follows.
		
		Suppose that $\widetilde B$ has a point $x$ in $B^d\setminus K_n^\beta$. 
		Let $C$ be the cap cut off by the half-space $H^+$ for which $x$ is the center of the $(d-1)$-dimensional ball $B^d\cap H$.
		If $K_n^\beta\subset H^-$, then $C$ does not contain any of the random points.
		If this is not the case, we show that there is an empty part of the cap $C$ with sufficiently large beta-content.
		We may assume that $x$ is in the direction of $\textbf{e}_d$.
		%the origin and $H$ is the hyperplane $x_d=0$.
		Thus, those coordinate hyperplanes of the standard Cartesian coordinate system that contain $x$ cut $C$ into $2^{d-1}$ identical pieces, and $B^d\setminus C$ into another $2^{d-1}$ larger pieces.
		If all of these parts contained a random point, then $x$ would be in their convex hull.
		Thus, there must be at least one empty coordinate corner, and its beta-content is at least $2^{-(d-1)}$ times that of the cap $C$, by symmetry..
		However, the definition of floating body yields $\Vol_d(C)\geq (c\log n/n)^{\frac{d+1}{d+1+2\beta}}$, and $\PP_\beta(C)\geq \gamma_1c \log n/n$. 
		%Let $c_2=c_1/2^{d-1}$.
		Then the probability of $T_n^c$ is at most $\left(1-\frac{\gamma_1}{2^{d-1}} c\log n/n\right)^n\leq n^{-\delta c}$ for some $\delta>0$, which depends only on $d$ and $\gamma_1$.
		In the uniform case, more precise asymptotics of probabilities are proven in \cite{BD97}.
		
		% \textcolor{red}{
			% Let $T_n$ be the event that the floating body} $B^d\left(v \ge (\kappa_d(c\log n/n))^{\frac{d+1}{d+1+2\beta}} \right)$ 
		% \textcolor{red}{ is contained in
			% $K_n$. Here $c=c_d$ is a large constant to be specified soon. We write $T_n^c$ for the
			% complement of $T_n$. We are going to use the main result of \cite{BD}
			% saying that there is a
			% constant $\delta$ depending only on $d$ such that $T_n^c$ occurs with probability $n^{-\delta
				% c}$.
			% }
		
		Finally, we choose $c$ to be sufficiently large, so the conditional expectation on the event $T_n^c$ can be omitted.
		
		We estimate the variance of the intrinsic volumes $V_s(K_n^\beta)$ from above using the Efron-Stein jackknife inequality \cite{ES81} and Kubota's formula
		\begin{equation}\label{Kubota}
			%\label{mixvolproj} 
			V_s(K)=C_{d,s} \int_{G(d,s)}\Vol_s (K|A) \nu_s(dA),
		\end{equation}
		where $K|A$ denotes the orthogonal projection of $K$ onto $A$ and $C_{d,s}$ is a constant depending only on $d$ and $s$.
		\begin{align}\label{ES-Kub}
			\var (V_s(K_n^\beta)) 
			%&\ll  n\cdot \Ex(V_s(K_{n+1}^\beta)-V_s(K_n^\beta))^2 \nonumber\\
			&\ll n\cdot \Ex[(V_s(K_{n+1}^\beta)-V_s(K_n^\beta))^2\indi(T_n)] \nonumber\\
			&\ll n\cdot \Ex \int_{G(d,s)}\int_{G(d,s)}
			\Vol_s\big((K_{n+1}^\beta|A)\setminus (K_n^\beta|A)\big) \times\nonumber\\
			&\qquad\qquad\times\Vol_s\big((K_{n+1}^\beta|B)\setminus (K_n^\beta|B)\big)\indi(T_n) \nu_s(dA)\nu_s(dB).
		\end{align}

		% The second term here is very small if the constant $c$ is chosen large enough because
		% $(V_s(K_{n+1})-V_s(K_n))^2 \le V_s(K_{n+1})^2\le V_s(K)^2$ and $\Ex(\indi(T_n^c)) \le n^{-\delta
			% c}$. We choose $c=c_d$ so large that the second term is smaller than the lower bound in
		%Theorem~\ref{thm:beta-main} proved in the previous section. 
		%So we concentrate on the first term:

		% \begin{align}\label{alap}
			% \var (V_s(K_n^\beta)) &\ll  n\cdot \Ex[(V_s(K_{n+1}^\beta)-V_s(K_n^\beta))^2\indi(T_n)]\notag\\
			% &\ll n\cdot \Ex \left [ \left (\int_{G(d,s)} \lambda_s(K_{n+1}^\beta\setminus K_n^\beta|A) \nu_s(dA) \right)\right.\times\notag\\
			% &\qquad\qquad\times \left.\left(\int_{G(d,s)} \lambda_s(K_{n+1}^\beta\setminus K_n^\beta|B) \nu_s(dB)\indi(T_n) \right )\right ]\notag\\
			% &\ll n\cdot \Ex \int_{G(d,s)}\int_{G(d,s)}
			% \lambda_s(K_{n+1}^\beta\setminus K_n^\beta|A) \times\notag\\
			% &\qquad\qquad\times\lambda_s(K_{n+1}^\beta\setminus K_n^\beta|B)\indi(T_n) \nu_s(dA)\nu_s(dB).
			% \end{align}
		
		Let $X_m$ denote the collection of the first $m$ random points $x_1,\ldots,x_m$, for $m\in\{1,\ldots,n+1\}$.
		For an index set $I=\{i_1,\ldots , i_s \} \subset \{1,\ldots, n\}$, let $F_I$ denote the convex hull of $x_{i_1}, \ldots, x_{i_s}$, which is an $(s-1)$-dimensional simplex with probability $1$. 
		Note that $(K_{n+1}^\beta|A)\setminus (K_n^\beta|A)$ is either empty or a union of simplices determined by $x_{n+1}|A$ and the facets of $K_n^\beta|A$ that can be seen from $x_{n+1}|A$.
		We say that a facet $F$ of $K_n^\beta|A$ is visible from a point $x\in A\setminus (K_n^\beta |A)$, if the open segment between $x$ and a point in the relative interior of $F$ is disjoint from $K_n^\beta|A$.
		%Let $\F_A(x_{n+1})$ denote the set of ($s-1$)-dimensional facets of $K_n^\beta|A$ that are visible from $x_{n+1}|A$.
		Let $U_A$ (or $U_B$) denote the event that the ($s-1$)-dimensional facet $F_I|A$ (or $F_J|B$, resp.), of $K_n^\beta|A$ is visible from $x_{n+1}|A$.

		The right-hand side of \eqref{ES-Kub} can be estimated as
		\begin{align}\label{nagyonhosszu}
			\eqref{ES-Kub}
			%&\leq n\cdot\Ex \left [ \int_{G(d,s)}\int_{G(d,s)}
			%\lambda_s(K_{n+1}^\beta\setminus K_n^\beta|A) \lambda_s(K_{n+1}^\beta\setminus
			%K_n^\beta|B) \nu_s(dA)\nu_s(dB)\indi(T_n)\right ]\notag\\
			&\ll n\cdot \int_{(B^d)^{n+1}} \int_{G(d,s)}\int_{G(d,s)} \Biggl
			(\,\sum_{F\in
				\F_A(x_{n+1})} \Vol_s([x_{n+1}|A, F])\Biggr )  \notag\\
			&\enspace\times  \Biggl (\,\sum_{F'\in \F_B(x_{n+1})} \Vol_s([x_{n+1}|B, F'])\Biggr )\indi(T_n) \nu_s(\d A)
			\nu_s(\d B) \prod_{i=1}^{n+1}f_{d,\beta}(x_i) \,\d X_{n+1}\notag\\
			%\end{align*}
			%\textcolor{red}{
				%By changing the order of integration and extending integration
				%over all index sets $I, J\in \binom{[n]}{s}$, we obtain the following.
				%}
			%\begin{align*}
			%&=n\cdot  \int_{G(d,s)}\int_{G(d,s)} \int_{(B^d)^{n+1}} \Biggl
			%(\sum_{I} \indi(F_I|A\in \F_A(x_{n+1}))\lambda_s([F_I, x_{n+1}]|A)\Biggr ) \times\notag\\
			%&\times  \Biggl (\sum_{J} \indi(F_J|B\in \F_B(x_{n+1}) )\lambda_s([F_J, x_{n+1}]|B)\indi(T_n)\Biggr )\times\notag\\
			%&\times \prod_{i=1}^{n+1}f_{d,\beta}(x_i) \,\d X_{n+1}\,\nu_s(\d A) \nu_s(\d B).
			&=n\cdot  \int_{G(d,s)}\int_{G(d,s)} \int_{(B^d)^{n+1}} \Biggl
			(\sum_{I} \indi(U_A)\Vol_s([F_I, x_{n+1}]|A)\Biggr ) \notag\\
			&\enspace \times   \!  \Biggl ( \! \sum_{J} \indi(U_B)\Vol_s([F_J, x_{n+1}]|B)\!\!\Biggr )\indi(T_n) \!\! \prod_{i=1}^{n+1}f_{d,\beta}(x_i) \,\d X_{n+1}\,\nu_s(\d A) \nu_s(\d B).
		\end{align}
		where the summation goes over all $s$-tuples $I$ and $J$.
		
		Similarly to \cite{BFV10}, we denote the smaller $s$- and $d$-dimensional caps cut off by the affine hull of $F_I|A$ or $A^\perp+\aff F_I $, resp. from $B^d$ by $C_s(I,A)$ and $C_d(I,A)$ respectively.
		For their volumes we use $V_s(I,A)=\Vol_s(C_s(I,A))$ and
		$V_d(I,A)=\Vol_d(C_d(I,A))$, moreover, for the probability content we use $\PP_\beta(I,A)=\PP_\beta(C_d(I,A))$.
		%\textcolor{red}{
			%We use the following notations. Let
			%$$C_s(I,A)=H_+(F_I|A)\cap B^d,$$ which is, in fact, a subset of the unit ball in the subspace $A$ and
			%$$C_d(I,A)=(H_+(F_I|A)+A^\perp)\cap B^d.$$
			%For the volumes of these caps we use $V_s(I,A)=\lambda_s(C_s(I,A))$ and
			%$V_d(I,A)=\lambda_d(C_d(I,A))$  } and for their probability content we use $\PP_\beta(I,A)=\PP_\beta(C_d(I,A))$. 
		
		The volumes of the simplices $[F_I, x_{n+1}]|A$ and $[F_J, x_{n+1}]|B$ can be estimated from above by the volumes of the caps $C_s(I,A)$ and $C_s(J,B)$, respectively.
		% \textcolor{red}{Now we are going to estimate these integrals from above using
			% the fact the simplices $[F_I, x_{n+1}]|A$ and $[F_J, x_{n+1}]|B$ are contained in the
			% associated caps $C_s(I,A)$ and $C_s(J,B)$, respectively.
			% }
		\begin{multline}\label{seged1}
			\eqref{nagyonhosszu}\ll n\cdot
			\int_{G(d,s)}\int_{G(d,s)} \sum_{I,J} \int_{(B^d)^{n+1}}
			%\indi(F_I|A\in \F_A(x_{n+1}))  \indi(F_J|B\in \F_B(x_{n+1}) )\times\\
			\indi(U_A)  \indi(U_B) V_s(I,A) V_s(J,B)\\
			\times  \indi(T_n)\prod_{i=1}^{n+1}f_{d,\beta}(x_i) \,\d X_{n+1}\,\nu_s(\d A) \nu_s(\d B).
		\end{multline}
		As the summation goes over all $s$-tuples $I$ and $J$, $I$ and $J$ may intersect.
		Let the number of common elements of $I$ and $J$ be $k$.
		
		Without loss of generality, we may choose $I=\{1,\dots,d\}$ and $J=\{s-k+1,\dots2s-k\}$, as the corresponding terms in \eqref{seged1} are independent of the choice of $i_1, \ldots, i_d$ and $j_1,\ldots, j_d$.
		%Let $I=\{1,\dots,d\}$ and $J=\{s-k+1,\dots2s-k\}$ for any given $k\in \{0,1, \ldots, s\}$. 
		%and let $F=\conv\{x_i:i\in I\}$ and $G=\conv\{x_j: j\in J\}$.
		%The corresponding terms in \eqref{seged1} are independent of the choice of $i_1, \ldots, i_d$ and $j_1,\ldots, j_d$.
		Thus, 
		% \textcolor{red}{
			% The summation extends over all $s$-tuples $I$ and $J$, so $I$ and $J$ may have nonempty
			% intersection. If we fix the size of $I\cap J$ to be $k$, say, then the corresponding terms in
			% the sum are clearly independent of the particular choice of $i_1, \ldots, i_s$ and $j_1,
			% \ldots, j_s$. For any given $k\in \{0,1, \ldots, s\}$ let $I=\{1,\dots,s\}$ and
			% $J=\{s-k+1,\dots2s-k\}$ and set $F=\conv\{x_i:i\in I\}$ and $G=\conv\{x_j: j\in J\}$. Thus
			% $I$ and $J$, and consequently $F$ and $G$ depend on $k$, but this is not shown in the
			% notation. We can estimate \eqref{seged1} from above by
			% }
		\begin{multline}\label{seged2}
			\eqref{seged1}\leq n\cdot \sum_{k=0}^s \binom{n}{s}
			\binom{s}{k} \binom{n-s}{s-k} \int_{G(d,s)}\int_{G(d,s)}
			\int_{(B^d)^{n+1}} \!\!\!\!\indi(U_A) \indi(U_B )\\
			\times V_s(I,A)V_s(J,B) \indi(T_n)\prod_{i=1}^{n+1}f_{d,\beta}(x_i) \,\d X_{n+1}\,\nu_s(\d A) \nu_s(\d B).
		\end{multline}
		
		By symmetry, we may assume the event $L$ that $V_s(I,A)\geq V_s(J,B)$. Therefore
		% \textcolor{red}{
			% Since the integrand is symmetric, we may restrict summation to those pairs of $F$ and $G$
			% where $V_s(I,A)\geq V_s(J,B)$, or equivalently, $V_d(I,A)\geq V_d(J,B)$, at
			% the price of a factor $2$. Thus, we can estimate \eqref{seged2} from above by
			% }
		\begin{multline*}
			\eqref{seged2}\ll  \sum_{k=0}^s n^{2s-k+1}  
			\int_{G(d,s)}\int_{G(d,s)} \int_{(B^d)^{n+1}}\indi(U_A)\indi(U_B)V_s(I,A)V_s(J,B)\\
			\times \indi(L)\indi(T_n)
			\prod_{i=1}^{n+1}f_{d,\beta}(x_i) \,\d X_{n+1}\,\nu_s(\d A) \nu_s(\d B).
		\end{multline*}
		
		Since $C_d(I,A)$ and $C_d(J,B)$ have at least $x_{n+1}$ in common, replacing $\indi(U_B)$ to $\indi(N):=\indi(C_d(I,A)\cap C_d(J,B)\neq \emptyset)$ in the integrand does not decrease the integral.
		We estimate the terms in the sum separately for each $k=0,\ldots,d$, denoted by $\Sigma_k$.
		% \textcolor{red}{
			% Let $\Sigma_k$ denote the $k$th term in this sum, $k=0,\ldots,s$. We are going to estimate
			% $\Sigma_k$ for each fixed $k$.
			% }
		
		% \textcolor{red}{
			% We first remove $\indi(G|B\in \F_B(x_{n+1}) )$ from the integrand in $\Sigma_k$ which clearly increases
			% the integral. We then multiply the integrand by $\indi(C_d(I,A)\cap C_d(J,B)\neq \emptyset)$.
			% This does not change the integral since the sets $C_d(I,A)$ and $C_d(J,B)$ have at least the
			% point $x_{n+1}$ in common. Thus we obtain the following
			% }
		\begin{multline*}
			\Sigma_k \ll n^{2s-k+1} \int_{G(d,s)}\int_{G(d,s)} \int_{(B^d)^{n+1}}
			\indi(U_A)V_s(I,A)V_s(J,B)\\
			\times \indi(N) \indi(L)\indi(T_n) \prod_{i=1}^{n+1}f_{d,\beta}(x_i) \,\d X_{n+1}\,\nu_s(\d A) \nu_s(\d B).
		\end{multline*}
		
		If $U_A$ holds, then all the points $x_{2s-k+1},\ldots,x_n$ must be contained in $B^d\setminus C(I,A)$ and $x_{n+1}$ has to be contained in $C(I,A)$.
		Then we can integrate with respect to $x_{2s-k+1},\ldots,x_n,x_{n+1}$, and the condition $T_n$ can be replaced by the condition $W_n=\{\PP_\beta(I,A)\le \gamma_2c\log n/n\}$, where the notation $\{\mathcal{S}\}$ for a probability event denotes the set of those outcomes of the random experiment for which the statement $\mathcal{S}$ holds.
		% \textcolor{red}{
			% Now, if $F|A\in\F_A(x_{n+1})$, then $x_{2s-k+1}, \ldots, x_n$ are all contained in $H_0(F,A)$ and
			% $x_{n+1}$ is contained in $H_+(F,A)$ because, under condition $T_n$, $C_d(I,A)$ is the
			% smaller cap cut off from $B^d$ by the hyperplane $A^{\perp}+\aff F$ and $o \in K_n$. We
			% integrate with respect to $x_{2s-k+1}, \ldots, x_n, x_{n+1}$, and the condition $T_n$ is
			% replaced by the condition $W_n$ saying that} $\PP_\beta(I,A)\le c\log n/n$.
		
		\begin{multline}\label{seged4}
			\Sigma_k \ll n^{2s-k+1} \int_{G(d,s)}\int_{G(d,s)}
			\int_{(B^d)^{2s-k}}
			(1-\PP_\beta(I,A))^{n-2s+k}\PP_\beta(I,A)\\
			\times  V_s(I,A)  V_s(J,B)\indi(N)\indi(L)\indi(W_n) \prod_{i=1}^{2s-k}f_{d,\beta}(x_i) \,\d X_{2s-k} \nu_s(\d A) \nu_s(\d B)
		\end{multline}
		
		We can use the geometric estimations in \cite{BFV10}*{p. 615} that do not depend on the probability distribution; 
		i.e., the conditions $N$ and $L$ imply that the cap $C_d(I,A)$ can be blown up from its center by a constant in order to contain $C_d(J,B)$ as well; 
		furthermore, the angle between the subspace $B$ and the center $z$ of the cap $C_d(I,A)$ has to be at most $b_d V_d(I,A)^{1/(d+1)}$ for a constant $b_d$ depending only on $d$.
		%We recall from \cite{BFV10} the fact that the conditions $C_d(I,A)\cap C_d(J,B)\neq \emptyset$ and $V_d(I,A)\geq V_d(J,B)$ imply that the cap $C_d(I,A)$ can be blown up from its center by a constant to contain $C_d(J,B)$ as well, therefore we may integrate with respect to the variables $x_i$, $i \in J$.
		Therefore, after integrating with respect to the variables $x_i$, $i \in J\setminus I$, we obtain that
		
		% \textcolor{red}{
			% In the next step, we integrate with respect to the variables $x_i$, $i \in J$.
			% \begin{multline}\label{pontokra}
				% \int_{(B^d)^{s-k}} \indi(C_d(I,A)\cap C_d(J,B)\neq
				% \emptyset)\indi(V_d(I,A)\geq V_d(J,B))\times\\
				% \times V_s(J,B) \indi(W_n)dx_{s+1}\ldots dx_{2s-k}.
				% \end{multline}
			% }
		
		% \textcolor{red}{
			% Since we assume that $V_d(I,A)\geq V_d(J,B)$, the height of the cap $C_d(I,A)$ is at least that of
			% $C_d(J,B)$. The condition $C_d(I,A)\cap C_d(J,B)\neq \emptyset$ implies that there is a
			% constant $\gamma$ such that $C_d(J,B)$ is contained in $\gamma C_d(I,A)$, where $\gamma
			% C_d(I,A)$ is an enlarged homothetic copy of $C_d(I,A)$, where the centre of homothety is $z
			% \in \partial B^d$, the centre of the cap $C_d(I,A)$ (cf \cite{Bar06}). Thus,
			% }
		
		% \begin{align}\label{pontokbecs}
			% &\int_{(B^d)^{s-k}} \indi(C_d(I,A)\cap C_d(J,B)\neq
			% \emptyset)\indi(V_d(I,A)\geq V_d(J,B))V_s(J,B)\times\notag\\
			% &\qquad\times  \indi(W_n)dx_{s+1}\ldots dx_{2s-k}\notag\\
			% &\ll \PP_\beta(I,A)^{s-k} V_s(I,A).
			% \end{align}

		% \textcolor{red}{
			% The conditions $C_d(I,A)\cap C_d(J,B)\neq \emptyset$  and $V_d(I,A)\geq V_d(J,B)$ can only be
			% satisfied if the angle, $\angle(z,B)$, between the vector $z$ and the subspace $B$ is not
			% larger than $2\alpha$, where $\alpha$ is the central angle of the cap $C_d(I,A)$. One can
			% easily verify that
			% \begin{equation}\label{alfakicsi}
				% \alpha\le  b_d V_d(I,A)^{1/(d+1)},
				% \end{equation}
			% where $b_d$ is a constant depending only on $d$. Using this condition on the mutual positions
			% of $z$ and $B$ together with \eqref{pontokbecs} we obtain that
			% }
		
		\begin{multline}\label{seged5}
			\eqref{seged4}\ll n^{2s-k+1}\int_{G(d,s)}\int_{G(d,s)}
			\int_{(B^d)^{s}}
			(1-\PP_\beta(I,A))^{n-2s+k}\PP_\beta(I,A)^{s-k+1} V_s(I,A)^2 \\
			\times  \indi\left(\angle(z,B)\leq 2b_d V_d(I,A)^{1/(d+1)}\right) \indi(W_n)\prod_{i=1}^{s}f_{d,\beta}(x_i) \,\d X_{s}\nu_s(\d A) \nu_s(\d B)
		\end{multline}
		
		Now we are going to estimate the inner integral for a fixed $A\in G(d,s)$ using the economic cap covering theorem.
		Due to the condition $W_n$, every cap $C_d(I,A)$ has beta-content at most $\gamma_2c\log n/n$.
		% \textcolor{red}{
			% We fix now $A\in G(d,s)$ and estimate
			% }
		% \begin{equation}\label{belso}
			% \int_{(B^d)^{s}} (1-\PP_\beta(I,A))^{n-2s+k} \PP_\beta(I,A)^{s-k+1} V_s(I,A)^2 \indi(W_n)\prod_{i=1}^{s}f_{d,\beta}(x_i) \,\d X_{s}.
			% \end{equation}
		
		We follow the standard argument: for each positive integer $h$ for which $2^{-h}\leq (c \log n)/n$, let $\mathcal{C}_h$ be a collection of caps $\{C_1, \ldots, C_{m(h)}\}$ forming the economic cap covering of the wet part of $B^d|A$ with $v=(2^{-h})^{\frac{s+1}{d+1+2\beta}}$ (we assume that $n$ is sufficiently large).
		For the $d$-dimensional caps $C_i'$ in $B^d$, whose projections to $A$ are $C_i$, we have that $\Vol_d(C_i')\ll (2^{-h})^{\frac{d+1}{d+1+2\beta}}$. 
		Consider an arbitrary $(x_1, \ldots, x_s)$ with the
		corresponding $C_d(I,A)$ having beta-content at most $\gamma_2c\log n/n$, and associate with
		$(x_1, \ldots, x_s)$ the maximal $h$ such that for some $C_i\in \mathcal{C}_h$, $C_s(I,A)\subset C_i$.
		Such an $h$ clearly exists.
		Then
		%\textcolor{red}{
			%We are going to use the Economic Cap Covering Theorem. 
			%Because of the condition $W_n$, every cap $C_d(I,A)$ has} beta-content 
		%\textcolor{red}{ at most $c\log n/n$. 
			%Let $h$ be a (positive) integer with $2^{-h}\leq (c \log n)/n$. 
			%For each such $h$, let $\M_h$ be a collection of caps $\{C_1, \ldots, C_{m(h)}\}$ forming the economic cap covering of the wet part of $B^s=B^d|A$ with} 
		%$t=(2^{-h})^{\frac{s+1}{d+1+2\beta}}$ 
		%\textcolor{red}{ (we suppose that $n$ is so large, that the theorem works). 
			%Each such cap $C_i$ is the projection of a $d$-dimensional cap $C_i(A)$ from
			%$B^d$ to $A$. Since the heights of $C_i$ and $C_i(A)$ are equal, we have that}
		%$\lambda_d(C_i(A))\ll (2^{-h})^{\frac{d+1}{d+1+2\beta}}$. 
		%\textcolor{red}{ Consider an arbitrary $(x_1, \ldots, x_s)$ with the
			%corresponding $C_d(I,A)$ having} beta-content \textcolor{red}{ at most $c\log n/n$, and associate with
			%$(x_1, \ldots, x_s)$ the maximal $h$ such that for some $C_i\in \M_h$, $C_s(I,A)\subset C_i$.
			%Such an $h$ clearly exists. It follows that}
		\begin{align*}
			V_s(I,A)\leq \Vol_s(C_i)\ll 2^{-h\frac{s+1}{d+1+2\beta}},\\
			V_d(I,A)\leq \Vol_d(C_i')\ll 2^{-h\frac{d+1}{d+1+2\beta}},
		\end{align*}
		%$$V_s(I,A)\leq \lambda_s(C_i)\ll 2^{-h\frac{s+1}{d+1+2\beta}},$$ %$$V_d(I,A)\leq \lambda_d(C_i(A))\ll 2^{-h\frac{d+1}{d+1+2\beta}},$$ 
		and $$\PP_\beta(C_i')\ll2^{-h}.$$
		
		On the other hand, by the maximality of
		$h$,
		\begin{align*}
			V_s(I,A)&\geq 2^{-(h+1)\frac{s+1}{d+1+2\beta}},\\
			V_d(I,A)&\geq 2^{-(h+1)\frac{d+1}{d+1+2\beta}},\\
			\PP_\beta(I,A)&\geq \gamma_2 2^{-(h+1)}.
		\end{align*}
		
		The integrand in \eqref{seged5} can be estimated by 
		%\textcolor{red}{ Now we shall integrate over $(B^d)^s$ under
			%condition $W_n$ by integrating each $(x_1,\ldots, x_s)$ on its associated $C_i(A)$, or more
			%precisely on $(C_i(A))^s$. The integrand in \eqref{belso} can be estimated as}
		\begin{multline*}(1-\PP_\beta(I,A))^{n-2s+k} \PP_\beta(I,A)^{s-k+1}
			V_s(I,A)^2\\
			\ll (1-\gamma_22^{-(h+1)})^{n-2s+k}\cdot 2^{-h(s-k+1)}\cdot
			2^{-2h\frac{s+1}{d+1+2\beta}},
		\end{multline*}
		and, after integrating each $(x_1,\ldots, x_s)$ on its associated $C_i'$, the inner integral is bounded by
		%\textcolor{red}{ Thus the integral on $(C_i(A))^s$ ($C_i(A)\in \M_h$) is bounded by}
		\begin{multline}\label{integrandh}
			\exp(-(n-2s+k)\gamma_22^{-(h+1)})2^{-h(s-k+1)}2^{-2h\frac{s+1}{d+1+2\beta}}(\PP_\beta(C_i'))^s\\
			\ll
			\exp(-(n-2s+k)\gamma_22^{-(h+1)})2^{-h(s-k+1)}2^{-2h\frac{s+1}{d+1+2\beta}}2^{-hs}.
		\end{multline}
		
		The final step in the proof is to calculate the number of elements of $\mathcal{C}_h$.
		The volume of the wet part of $B^s$ (a spherical shell) with parameter $2^{-h\frac{s+1}{d+1+2\beta}}$ is $\Vol_s(B^s(2^{-h\frac{s+1}{d+1+2\beta}}))\approx 2^{\frac{-2h}{d+1+2\beta}}$.
		(The $\approx$ notation makes sense, since $h\to \infty$ as $n\to \infty$). 
		Therefore,
		%\textcolor{red}{
			%Now we return to \eqref{seged5}. In order to estimate the integral, we still need the number %of elements $|\mathcal{C}_h|$ of $\mathcal{C}_h$. The volume of the wet part of $B^s$ with parameter}
		%$2^{-h\frac{s+1}{d+1+2\beta}}$ is $\lambda_s(B^s(2^{-h\frac{s+1}{d+1+2\beta}}))\approx 2^{\frac{-2h}{d+1+2\beta}}$
		%\textcolor{red}{ (the $\approx$ notation makes sense, since $h\to \infty$ as $n\to \infty$). It readily follows that}
		$$|\mathcal{C}_h|\ll \frac{2^{-2h\frac{1}{d+1+2\beta}}}{2^{-h\frac{s+1}{d+1+2\beta}}}=2^{-h\frac{(1-s)}{d+1+2\beta}}.$$
		
		Assembling the pieces estimating the two inner integrals of \eqref{seged5},  together with the condition $\angle(z,B)\leq 2b_d V_d(I,A)^{1/(d+1)}$ and applying Lemma \ref{szoges} and \eqref{integrandh}, we obtain with $h_0=\left \lfloor c \log n/n \right \rfloor$,
		%\textcolor{red}{
			%Keeping in mind the condition $\angle(z,B)\leq 2b_d V_d(I,A)^{1/(d+1)}$ and applying Lemma \ref{szoges} and
			%\eqref{integrandh}, we obtain with $h_0=\left \lfloor \frac{c \ln n}{n} \right \rfloor$, }
		\begin{align}
			\nonumber & \int_{G(d,s)} \int_{(B^d)^{s}} (1-\PP_\beta(I,A))^{n-2s+k}
			\PP_\beta(I,A)^{s-k+1} V_s(I,A)^2 \times\\
			\nonumber &\times \indi\left(\angle(z,B)\leq 2b_d V_d(I,A)^{1/(d+1)}\right)
			\prod_{i=1}^{s}f_{d,\beta}(x_i)\,\d X_{s} \nu_s(\d B)\\
			\nonumber &\ll  \sum_{h=h_0}^{\infty} \exp(-(n-2s+k)\gamma_2 2^{-(h+1)})2^{-h(s-k+1)}2^{-2h\frac{s+1}{d+1+2\beta}}2^{-hs}\times \\
			\nonumber & \times  |\mathcal{C}_h| \, \nu_s(\{B \, | \, \angle(z,B)< 2b_d 2^{\frac{-h}{d+1+2\beta}}\})\\
			%\nonumber & \ll & \sum_{h=h_0}^{\infty} \exp(-(n-2s+k)2^{-h-1})2^{-h(s-k+1)}2^{-2h\frac{s+1}{d+1+2\beta}}2^{-hs} 2^{-h\frac{(1-s)}{d+1+2\beta}} 2^{\frac{-h(d-s)}{d+1+2\beta}}\\
			&\ll \sum_{h=h_0}^{\infty} \exp(-(n-2s+k)\gamma_2 2^{-(h+1)})2^{-h[2s-k+1+\frac{d+3}{d+1+2\beta}]}.\label{sum}
		\end{align}
		
		Calculating the sum of the series is very similar to \cite{BFV10}*{p. 617}, with slightly different exponents. 
		The order of magnitude of the sum in \eqref{sum} is $n^{-2s+k-1}n^{-\frac{d+3}{d+1+2\beta}}$, therefore
		
		% \textcolor{red}{
			% Now, we divide the sum in \eqref{integrandh} into two parts. First, let $h_1$ be defined by $$2^{-h_1}\leq \frac1n < 2^{-h_1+1}.$$ Since in this case $\exp(-(n-2s+k)2^{-h-1})$ is smaller than $1$, it follows that }
		% \begin{multline}\label{farok}
			% \sum_{h=h_1}^{\infty} \exp(-(n-2s+k)2^{-h-1})2^{-h[(2s-k+1)+\frac{d+3}{d+1+2\beta}]}\\
			% \leq\sum_{h=h_1}^{\infty} 2^{-h[(2s-k+1)+\frac{d+3}{d+1+2\beta}]}\ll n^{-2s+k-1}n^{-\frac{d+3}{d+1+2\beta}}.
			% \end{multline}
		
		% \textcolor{red}{ For the other part, when $h_0\leq h<h_1$, we let $\ell=h_1-h$. Then $\ell$ runs from $1$ to
			% $\ell_1=h_1-h_0$. }
		% \begin{eqnarray}\label{torzs}
			% \nonumber & & \sum_{h=h_0}^{h_1-1} \exp(-(n-2s+k)2^{-h-1})2^{-h[(2s-k+1)+\frac{d+3}{d+1+2\beta}]}\\
			% \nonumber & \leq & \sum_{\ell=1}^{\ell_1} \exp(-(n-2s+k)2^{-h_1+\ell-1})2^{-(h_1-\ell)
				% [(2s-k+1)+\frac{d+3}{d+1+2\beta}]}\\
			% \nonumber & \ll & \sum_{\ell=1}^{\ell_1} \exp(-(n-2s+k)2^{-h_1+\ell-1}) n^{-(2s-k+1)}
			% 2^{\ell(2s-k+1)}n^{-\frac{d+3}{d+1+2\beta}}2^{\ell \frac{d+3}{d+1+2\beta}}\\
			% \nonumber & \ll & n^{-(2s-k+1)}n^{-\frac{d+3}{d+1}} \sum_{\ell=1}^{\infty} \exp (-2^\ell)
			% \cdot 2^{\ell[(2s-k+1)+\frac{d+3}{d+1+2\beta}]}\\
			% & \ll & n^{-(2s-k+1)}n^{-\frac{d+3}{d+1+2\beta}} \sum_{j=1}^{\infty} \exp (-j) j^{4d}\ll n^{-(2s-k+1)}
			% n^{-\frac{d+3}{d+1+2\beta}}.
			% \end{eqnarray}
		
		%\textcolor{red}{
			%Now, putting \eqref{farok} and \eqref{torzs} back to \eqref{seged5}, we get that}
		\begin{equation*}
			\Sigma_k \ll n^{2s-k+1}\int_{G(d,s)} n^{-(2s-k+1)}n^{-\frac{d+3}{d+1+2\beta}}\nu_s(dA)\ll
			n^{-\frac{d+3}{d+1+2\beta}}.
		\end{equation*}
		%\textcolor{red}{ Summing this for all $k=0,\ldots,s$ proves the upper bound in Theorem \ref{thm:beta-main}. 
			%}
		This holds for all $k=0,\ldots,d$, thus the sum of them is of the same order of magnitude, which proves the theorem.

		\section{Proof of the lower bound in Theorem~\ref{thm:beta-main}}\label{sec:beta-lower}
		
		The proof of the lower bound essentially follows a method presented by Reitzner \cite{R05}: one defines small caps independently of each other and provides a lower bound on the variance in each cap. This method has been used in various settings; see, for instance, \cite{BFV10}, \cite{BT20}, and \cite{FGV22}. 
		As the geometry of the intrinsic volumes and $B^d$ is the same as in \cite{BFV10}*{Section 4}, we only highlight the differences that are due to in the probability model.
		
		% \textcolor{red}{
			% The idea of the proof of the lower bound is similar to those presented in \cite{R05} and
			% \cite{BFRV09}, namely, we define small independent caps, and we show that the variance is
			% ``large'' in each cap. From the properties of the variance the required estimate will follow.
			% }
		
		For a very small $t>0$, we take a cap $C(x,t)$ in $B^d$ of height $t$, whose center is $x\in\Sp^{d-1}$.
		Let the hyperplane that cuts off $C(x,t)$ be $H(x,t)$.
		Then we consider a regular $(d-1)$-simplex inscribed into $B(x,t)=B^d\cap H(x,t)$, call its vertices $w_1,w_2,\ldots,w_d$.
		Furthermore, let $w_0=x$.
		Then $\Delta=[w_0, w_1, \ldots, w_d]$ is a $d$-dimensional simplex
		inscribed in $C(x,t)$. 
		We define the smaller simplices as
		$$\Delta_j=\Delta_j(x,t)=w_j+\frac{1}{4d}([w_0, w_1, \ldots,
		w_d]-w_j)$$
		for $j=0,1,\ldots,d$.
		Each $\Delta_j$ is a homothetic copy of $\Delta$ obtained by a homothety with center at $w_j$ and factor $1/(4d)$, see Figure~\ref{abra:delta_j}.
		We need to estimate the beta-content of $\Delta_j$.
		\begin{lemma}\label{szimplex-beta}
			For $j=0,\ldots,d$,
			\begin{equation*}
				\PP_\beta(\Delta_j(x,t))\approx t^{\frac{d+1+2\beta}{2}}, 
				\,\text{ as }t\to 0^+.
			\end{equation*}
		\end{lemma}
		
		\begin{figure}
			\centering
			\includegraphics[width=0.75\linewidth]{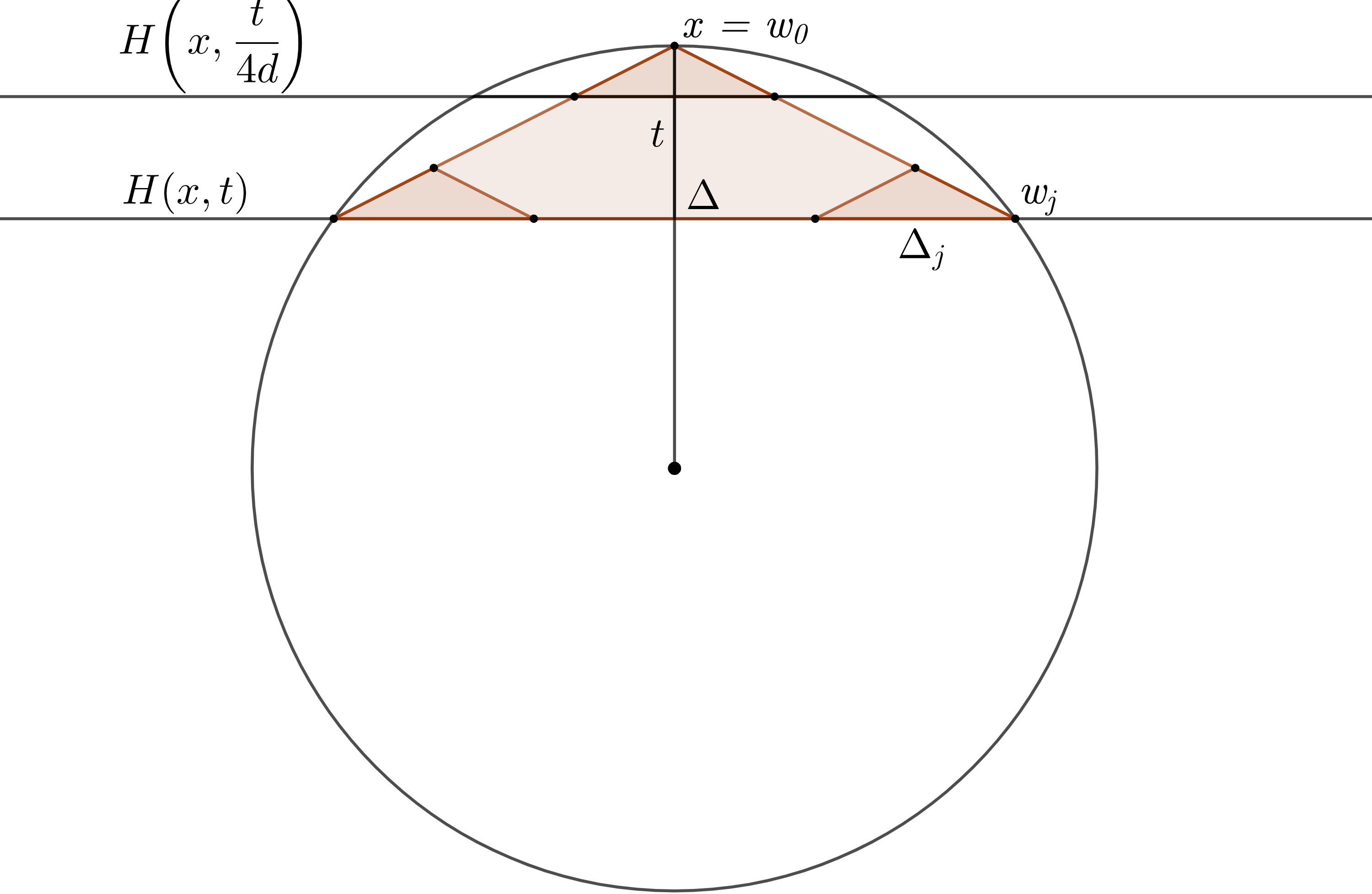}
			\caption{The construction of the small simplices}
			\label{abra:delta_j}
		\end{figure}
		
		\begin{proof} 
			On the one hand, each $\Delta_j$ is contained in a cap which has a probability content of order $t^{\frac{d+1+2\beta}{2}}$, hence the upper bound.
			On the other hand, we prove the lower bound for the case $j=0$, and indicate the modifications required for the other cases.
			%we divide the proof of the lower bound into two parts based on the sign of $\beta$.

			% \textcolor{blue}{
				% First, let $\beta\geq0$.
				% By \cite{BFV10}, $\Vol_d(\Delta_j)\approx t^{\frac{d+1}{2}}$. 
				% For a fixed $j$, let $D$ be the cap of the same volume as $\Delta_j$, which has the largest intersection with $\Delta_j$.
				% Then for the height $t'$ of $D$, we have $t'\approx t$, and $\PP_\beta(D)\approx t^{\frac{d+1+2\beta}{2}}$ by Lemma~\ref{sapka-PC}.
				% Apart from the intersection of $D$ and $\Delta_j$, the rest of the simplex $\Delta_j$ is contained in a smaller concentric ball that is disjoint from the rest of the cap $D$.
				% Since $\beta\geq0$, the density function $f_{d,\beta}$ is larger in this smaller ball.
				% Thus, $\PP_\beta(\Delta_j(x,t))\geq\PP_\beta(D)$.
				% }

			%Now, let $\beta<0$. 
			We are going to estimate the integral
			$\int_{\Delta_0}(1-\|x\|^2)^\beta\,\d x$.
			We may assume that the apex $w_0$ of the simplex is $(0,\ldots,0,1)$, and the base is a regular $(d-1)$-dimensional simplex inscribed in a $(d-1)$-dimensional ball of radius $\sqrt{\frac{t}{4d}(2-\frac{t}{4d})}$ centered at $(0,\ldots,0,1-\frac{t}{4d})$ in the hyperplane $x_d=1-\frac{t}{4d}$.
			We denote the points in $\Delta_0$ by $x=(x',y):=(x_1,\ldots,x_{d-1},y)$, i.e. $x'\in\R^{d-1}$, to emphasize the role of the $d$-th coordinate.
			Then $\|x\|^2=y^2+\|x'\|^2$, and 
			\begin{equation}\label{s-lemma:1}
				\int_{\Delta_0}(1-\|x\|^2)^\beta\,\d x = \int_{1-\frac{t}{4d}}^1\int_{\delta \cdot\Delta_{d-1}\big(\sqrt{\frac{t}{4d}(2-\frac{t}{4d})}\big)} (1-y^2-\|x'\|^2)^\beta\,\d x'\d y,
			\end{equation}
			where $\Delta_{d-1}\Big(\sqrt{\frac{t}{4d}(2-\frac{t}{4d})}\Big)$ is the regular $(d-1)$-dimensional simplex inscribed in the $(d-1)$-dimensional ball of radius $\sqrt{\frac{t}{4d}(2-\frac{t}{4d})}$; 
			and $\delta=\delta(y)=\frac{4d}{t}(1-y)$ is the linear factor of homothety in $y$.
			
			We apply the substitution $y=1-u$, then
			\begin{equation}\label{s-lemma:2}
				\eqref{s-lemma:1}= \int_0^{\frac{t}{4d}}\int_{u\frac{4d}{t} \cdot\Delta_{d-1}\big(\sqrt{\frac{t}{4d}(2-\frac{t}{4d})}\big)} (2u-u^2-\|x'\|^2)^\beta\,\d x'\d u.
			\end{equation}
			Then we use the fact that the height of $\Delta_j$ is of order $t$, and the linear size of the base is proportional to $\sqrt{t}$, by substituting $u=\frac{vt}{4d}$ and $x'=w\sqrt{\frac{t}{4d}}$.
			Then
			\begin{align}\label{s-lemma:3}
				\eqref{s-lemma:1}&= \int_0^1\int_{v \cdot\Delta_{d-1}\big(\sqrt{2-\frac{t}{4d}}\big)} \Big(\frac{vt}{2d}-\frac{v^2t^2}{(4d)^2}-\frac{t}{4d}\|w\|^2\Big)^\beta \frac{t}{4d}\sqrt{\frac{t}{4d}}^{d-1}\d w\d v \notag \\
				&=\frac{t^{\frac{d+1+2\beta}2}}{(4d)^{\frac{d+1+2\beta}2}} \int_0^1\int_{v \cdot\Delta_{d-1}\big(\sqrt{2-\frac{t}{4d}}\big)} \Big(2v-\frac{v^2t}{4d}-\|w\|^2\Big)^\beta \,\d w\d v.
			\end{align}
			Since $t\to0$, we may assume that $t\leq4d\varepsilon_1$, for some constant $\varepsilon_1$. 
			Then $\Delta_{d-1}\Big(\sqrt{2-\frac{t}{4d}}\Big)\supseteq \Delta_{d-1}\big(\sqrt{2-\varepsilon_1}\big)$. Since the integrand, inherited from the definition of the beta distribution, is non-negative, restricting the integration to $v\cdot\Delta_{d-1}\big(\sqrt{2-\varepsilon_1}\big)$ does not increase the integral.
			Now we divide the proof into two parts based on the sign of $\beta$.
			
			First, let $\beta<0$.
			Then, estimating $\Big(2v-\frac{v^2t}{4d}-\|w\|^2\Big)$ from above gives a lower bound for the integrand.
			Therefore,
			\begin{equation}\label{s-lemma:4}
				\eqref{s-lemma:3}\geq \frac{t^{\frac{d+1+2\beta}2}}{(4d)^{\frac{d+1+2\beta}2}} \int_0^1\int_{v \cdot\Delta_{d-1}(\sqrt{2-\varepsilon_1})}  \big(2v-\|w\|^2\big)^\beta \,\d w\d v.
			\end{equation}
			The integral in $\eqref{s-lemma:4}$ is now independent of $t$, hence the right-hand side is a positive constant times $t^{\frac{d+1+2\beta}2}$.
			
			Now, let $\beta\geq0$.
			In this case, we need to estimate the integrand in \eqref{s-lemma:3} from below. Thus we obtain
			\begin{equation*}
				%\label{s-lemma:5}
				\eqref{s-lemma:3}\geq \frac{t^{\frac{d+1+2\beta}2}}{(4d)^{\frac{d+1+2\beta}2}} \int_0^1\int_{v \cdot\Delta_{d-1}(\sqrt{2-\varepsilon_1})} \big(2v-v^2\varepsilon_1-\|w\|^2\big)^\beta \,\d w\d v,
			\end{equation*}
			where the integral is again independent of $t$. 
			Since $\|w\|^2\leq v^2(2-\varepsilon_1)$, the integrand is bounded from below by $2(v-v^2)$, which is positive for $v\in(0,1)$. 
			Thus, the right-hand side is a positive constant times $t^{\frac{d+1+2\beta}2}$.
			
			The cases $j=1,\ldots,d$ are similar, relying on the same scaling argument; however, the integration domain in \eqref{s-lemma:1} needs to be modified in the following way.
			The $(d-1)$-dimensional ball of radius $\sqrt{\frac{t}{4d}(2-\frac{t}{4d})}$ circumscribed of the regular $(d-1)$-simplex is not centered, but translated such that one vertex of the simplex lies on the boundary of the $d$-dimensional unit ball.
			Furthermore, the homothety, which is linear in $y$, also changes the center of $\Delta_{d-1}$ together with its size.
			Moreover, $y$ varies between $1-t$ and $1-t+\frac{t}{4d}$, and in \eqref{s-lemma:2}, we substitute $u=1-t+\frac{t}{4d}-y$.
			
		\end{proof}

		We choose a point $z_j$ in each $\Delta_j(x,t)$, fix $x$, $t$ and $z_j \in \Delta_j(x,t)$ for $j=1,\ldots, d$, and write $F=[z_1, \ldots, z_d]$. 
		Similarly to \cite{BFV10}, we define the function $\hat V_s : \Delta_0(x,t) \to \mathbb{R}$ as follows
		$$\hat V_s(z_0)= \int_{L
			\in G(d,s),\\ L\cap \Sigma_2\neq \emptyset} \Vol_s([z_0,F] | L) \nu_s(dL),$$ 
		where $\Sigma_2(x,t)=\Sp^{d-1} \cap \left (x + 2d \sqrt tB^d \right ).$

		\begin{lemma}%\label{alsokulcs}
			If $Z$ is a random point chosen according to the normalized beta-distribution on $\Delta_0(x,t)$,
			then 
			$${\rm Var} \;\hat V_s(Z)\gg t^{d+1}.$$
		\end{lemma}
		For the proof, we refer to \cite{BFV10}*{Lemma~2}, as it is mainly geometric and the distribution does not make a difference.

		We set 
		% \textcolor{red}{
			% It is sufficient to prove the lower bound for large enough $n$. We
			% fix}
		\begin{equation*}
			%\label{eq:tndef}
			t_n=n^{-\frac2{d+1+2\beta}},
		\end{equation*}
		then, by Lemma~\ref{sapka-PC}, $\PP_\beta(C(x,t_n))\approx 1/n$ for all $x\in \Sp^{d-1}$.
		Let $y_1,\ldots,y_m\in \Sp^{d-1}$ be a maximal set of points such that $d(y_i,y_j)\geq 2\sqrt{\gamma}\sqrt{t_n}$ for $i,j\in\{1,\ldots,m\}$ and for a constant $\gamma$ specified as follows.
		Let $\gamma$ be so large that the caps $C(y_j,\gamma t_n)$ ($j=1,\ldots,m$) are pairwise disjoint. Then
		\begin{equation*}
			%\label{msize} 
			m\gg n^{\frac{d-1}{d+1+2\beta}}.
		\end{equation*}
		% \textcolor{red}{
			% and hence} 
		%$V(C(x,t_n))\approx 1/n$ 
		%$\PP_\beta(C(x,t_n))\approx 1/n$
		% \textcolor{red}{
			%for all $x\in S^{d-1}$. 
			% We choose a
			% maximal family of points $y_1,\ldots,y_m\in S^{d-1}$ such that for $i\neq
			% j$, we have
			% $$
			% \|y_i-y_j\|\geq 2\sqrt{\gamma}\sqrt{t_n}.
			% $$
			% This condition implies that the caps $C(y_j,\gamma T_n)$ ($j \in [m]$) are disjoint. One can see that}
		% \begin{equation}
			% \label{msize} m\gg n^{\frac{d-1}{d+1+2\beta}}.
			% \end{equation}
		
		%\textcolor{red}{
			%For each $j \in [m]$ we construct the simplex $\Delta (y_j,t_n)$ in the cap $C(y_j,t_n)$ and for each $i=0,1,\dots,d$ we construct the corresponding small simplices $\Delta_i(y_j,t_n)$.
			%For $j\in [m]$, let $A_j$ denote the event that each $\Delta_i(y_j,t_n)$, $i=0,\ldots,d$
			%contains exactly one random point out of $x_1,\ldots,x_n$, and $C(y_j,\gamma t_n)$ contains no other random point. 
			%We note that the definition of $\Delta_i$, \eqref{capprop} and
			%\eqref{capind} yield that for $i=0,\ldots,d$, we have}
		
		Inscribe the simplex $\Delta (y_j,t_n)$ in the cap $C(y_j,t_n)$ for each $j \in \{1,\ldots,m\}$, and construct the small simplices $\Delta_i(y_j,t_n)$.
		Let $A_j$ denote the event that each $\Delta_i(y_j,t_n)$, $i=0,\ldots,d$, contains exactly one random point out of $x_1,\ldots,x_n$, and no other random points are contained in $C(y_j,\gamma t_n)$.
		By Lemmas~\ref{sapka-PC} and \ref{szimplex-beta}, we have
		%$$
		%V(\Delta_i(y_j,t_n))\gg n^{-\frac{d+1}{d+1+2\beta}} \text{ and }V(C(y_j,\gamma t_n))\ll n^{-\frac{d+1}{d+1+2\beta}}.
		%$$  
		%and \todo{ezt kéne belátni még}
		$$\PP_\beta(\Delta_i(y_j,t_n))\gg 1/n \text{ and } \PP_\beta(C(y_j,\gamma t_n))\ll
		1/n.$$
		Therefore, for $j=1,\ldots,m$, it holds that
		\begin{equation*}
			%\label{Ajprob} 
			\mathbb{P}(A_j)\gg \binom {n} {d+1}
			\left(\frac{1}n\right)^{d+1}\left(1-\frac{1}n\right)^{n-d-1}\gg 1.
		\end{equation*}

		Similarly to \cite{BFV10}, we obtain
		\begin{equation*}
			\var V_s(K_n^\beta)\gg m \cdot\PP(A_j)\cdot t_n^{d+1} 
			%\gg n^{\frac{d-1}{d+1+2\beta}}\cdot n^{-\frac2{d+1+2\beta}}
			\gg n^{-\frac{d+3}{d+1+2\beta}}.
		\end{equation*}

		\section{Proof of Theorem~\ref{beta-CHT}}\label{sec:CHT}
		The proof essentially follows the method of Th\"ale, Turchi and Wespi \cite{TTW18}; Besau and Th\"ale \cite{BT20} adapted it to non-Euclidean geometries, and Fodor and Papv\'ari \cite{FP24} used a similar argument for a spindle convex model in $\R^2$.
		We apply the normal approximation bounds that were used by \cite{Cha08} and \cite{LRP17}.
		
		%For a Polish space $E$ (in our case, $E$ is the open unit ball in $\R^d$), let $f: \bigcup_{k=1}^n E^k\to\R$ be a measurable and symmetric function that acts on point configurations of at most $n\in\N$ points in $E$.
		Let $f: \bigcup_{k=1}^n (B^d)^k\to\R$ be a measurable and symmetric function that acts on point configurations of at most $n\in\N$ points in $B^d$.
		For $x=(x_1,\ldots,x_n)\in (B^d)^n$, we define the first- and second-order difference operators applied to $f(x)=f(x_1,\ldots,x_n)$ as
		$$D_if(x):=f(x)-f(x^i) \qquad\text{and}\qquad 
		D_{i,j}f(x):=f(x)-f(x^i)-f(x^j)+f(x^{i,j}),$$
		respectively, where $x^i$ denotes the $(n-1)$-dimensional vector one gets by removing the $i$-th coordinate of $x$, and similarly, $x^{i,j}$ arises by removing the $i$-th and $j$-th coordinates of $x$.
		
		For a random vector $X=(x_1,\ldots,x_n)$ of elements of $B^d$, we introduce the random copies $X'$ and $X''$ of $X$. 
		Let the random vector $Z=(z_1,\ldots,z_n)$ be a recombination of $\{X,X',X''\}$, i.e., for all $i\in\{1,\ldots, n\}$, $z_i\in\{x_i,x'_i,x''_i \}$ .
		
		We need the following quantities:
		\begin{align*}
			Q_1&:=\sup_{(Y,Y',Z,Z')}\Ex\Big[\indi(D_{1,2}f(Y)\neq 0)\indi(D_{1,3}f(Y')\neq 0) \big(D_2f(Z)\big)^2\big(D_3f(Z')\big)^2\Big],\\
			Q_2&:=\sup_{(Y,Z,Z')}\Ex\Big[\indi(D_{1,2}f(Y)\neq 0)\big(D_1f(Z)\big)^2\big(D_2f(Z')\big)^2\Big],\\
			Q_3&:=\Ex\big[\left|D_1f(X)\right|^4\big],\\
			Q_4&:=\Ex\big[\left|D_1f(X)\right|^3\big],
		\end{align*}
		where the suprema in the first two definitions are taken over all $4$- and $3$-tuples of vectors $(Y,Y',Z,Z')$ and $(Y,Z,Z')$, respectively, that are recombinations of $\{X,X',X''\}$.
		
		Let $W=f(x_1,\ldots,x_n)$ and assume that $W$ is centered, with finite positive second moment. 
		Let $G$ denote a standard Gaussian random variable.
		Then, by \cite{LRP17},
		\begin{equation}\label{normal-bound}
			d_W\Bigg(\frac{W}{\sqrt{\var W}},G\Bigg)\ll \frac{\sqrt n}{\var W} \Bigg(
			\sqrt{n^2 Q_1}+\sqrt{n Q_2}+\sqrt{Q_3}\Bigg)+\frac{n}{\big(\var W\big)^\frac32}Q_4.   
		\end{equation}
		
		In our case, for a fixed $\beta>-1$, let $x_1,\ldots,x_n$ be i.i.d. random points according to the beta-distribution, and for a fixed $s\in\{1,\ldots,d\}$, let $W=V_s(K_n^\beta)-\Ex V_s(K_n^\beta)$.
		Then $\Ex W=0$, and by Theorem~\ref{thm:beta-main}, the variance is finite.
		
		As in Section~\ref{sec:beta-upper}, let $T_n$ denote the event that the floating body $\widetilde B=B^d[(c\log n/n)^{\frac{d+1}{d+1+2\beta}}]$ is contained in $K_n^\beta$ for a suitable constant $c$.
		Then there is a constant $\delta$ such that the probability of $T_n^c$ is at most $n^{-\delta c}$.
		Similarly, let $\tau_n$ denote the event that $\widetilde B$ is contained in $\bigcap_{W\in\{Y,Y',Z,Z'\}}[W_4,\ldots,W_n]$.
		The probability of $\tau_n^c$ is smaller than $n^{-\delta'c}$ for a constant $\delta'$.
		
		The main difference between our model and the model studied in \cite{TTW18} is due to the probability distribution of the random points, and not in the geometry.
		% we can rely on some statements regarding the unit ball in $\R^d$.
		
		We estimate the first-order difference operator using Kubota's formula \eqref{Kubota}:
		\begin{equation}\label{diff-Kub}
			D_1V_s(K_n^\beta)=C_{d,s}\int_{G(d,s)}\Vol_s (K_n^\beta|A\setminus K_{n-1}^\beta|A) \indi(X_1\in B^d\setminus \widetilde B) \nu_s(dA),
		\end{equation}
		where $K_{n-1}^\beta$ now denotes $[x_2,\ldots,x_n]$.
		First, we need to calculate the probability of the event that, assuming $T_n$, $x_1$ is not in $K_{n-1}^\beta$, because otherwise the integrand in \eqref{diff-Kub}  is zero.
		Assuming $T_n$, we have the following.
		\begin{equation}\label{vizes-resz}
			\PP(x_1\in B^d\setminus K_{n-1}^\beta)= \PP_\beta(B^d\setminus K_{n-1}^\beta)\ll\PP_\beta(B^d\setminus\widetilde B)
		\end{equation}
		Therefore, we need to estimate the beta-content of the wet part of $B^d$ with parameter $v=(c\log n/n)^{\frac{d+1}{d+1+2\beta}}$. Recall that the beta content of a cap $C$ with volume at most $v$ is at most $\gamma_2c\log n/n$.
		
		%For the height $t$ of a cap $C$ of volume $v$, we have that $t\approx v^{\frac2{d+1}}\approx\Big(\frac{\log n}{n}\Big)^{\frac{2}{d+1+2\beta}}$, so the beta-content of $C$ is $\PP_\beta(C)\approx\frac{\log n}{n}$.
		
		Theorem \ref{R-ecct} (economic cap covering theorem) states that the wet part $B^d\setminus\widetilde B$ can be covered with caps of volume $v$, and the number of caps needed is of order at most $v^{-\frac{d-1}{d+1}}$.
		We can estimate the probability content of the wet part from above by taking the probability content of all the caps used in the covering. Thus,
		\begin{equation*}
			\eqref{vizes-resz}\ll v^{-\frac{d-1}{d+1}}\PP_\beta(C) \ll \Big(\frac{\log n}{n}\Big)^{\frac{2+2\beta}{d+1+2\beta}}.
		\end{equation*}
		
		Now, let $z\in\Sp^{d-1}$ be the closest point to $x_1$ on the boundary of $B^d$, and define its visibility region as
		$$\Vis_z(n):=\{x\in B^d\setminus \widetilde B:[x,z]\cap\widetilde B=\emptyset\}.$$
		Then for the diameter of the visibility region, it holds that $\diam \Vis_z(n)\ll (\log n/n)^{\frac1{d+1+2\beta}}$ (cf. \cite{Vu05}*{Lemma 6.2}).
		Similarly to \cite{TTW18}, we construct the cap $C$, which is the convex hull of all the points of $\Sp^{d-1}$, whose distance from $z$ is at most $\diam \Vis_z(n)$.
		For the central angle $\alpha$ of $C$, we have $\alpha\ll\big(\frac{\log n}{n}\big)^{\frac1{d+1+2\beta}}$ and the integrand in \eqref{diff-Kub} can only be positive if $\angle(z,A)\ll\alpha$.
		By Lemma~\ref{szoges}, $\nu_s\{A\in G(d,s)\, | \, \angle(z,A)\ll\big(\frac{\log n}{n}\big)^{\frac1{d+1+2\beta}}\}\ll \big(\frac{\log n}{n}\big)^{\frac{d-s}{d+1+2\beta}}$.
		Although the projection of $C$ onto a subspace $A\in G(d,s)$ is generally not a cap, a simple trigonometric calculation shows that, assuming $\angle(z,A)\ll\alpha$, $C|A$ is contained in a cap whose height is of the same order as the height of $C$. 
		We can estimate the volume of $C$ and the $s$-dimensional volume of its projection as
		$$\Vol_d(C)\ll\Big(\frac{\log n}n\Big)^{\frac{d+1}{d+1+2\beta}} \quad \text{ and }\quad \Vol_s(C|A)\ll\Big(\frac{\log n}n\Big)^{\frac{s+1}{d+1+2\beta}}.$$
		
		Since $(K_n^\beta|A)\setminus (K_{n-1}^\beta|A)\subset C|A$, \eqref{diff-Kub} can be estimated from above as
		\begin{align}\label{diff}
			D_1V_s(K_n^\beta)&\ll\Vol_s(C|A)\nu_s\{A\in G(d,s)\, | \, \angle(z,A)\leq\alpha\}\indi(x_1\in B^d\setminus\widetilde B)\notag\\ 
			&\ll \Big(\frac{\log n}n\Big)^{\frac{d+1}{d+1+2\beta}}\indi(x_1\in B^d\setminus\widetilde B).
		\end{align}
		Therefore, we estimate the expectation of $\big(D_1V_s(K_n^\beta)\big)^p$ as follows.
		In the definition of $\widetilde B$, we choose $c$ to be sufficiently large, so the conditional expectation on the event $T_n^c$ can be omitted.
		For $p\in\{1,\ldots,4\}$,
		\begin{equation}\label{diff-ex}
			\Ex[\big(D_1V_s(K_n^\beta)\big)^p]\ll\Big(\frac{\log n}n\Big)^{\frac{p(d+1)}{d+1+2\beta}}\PP_\beta(B^d\setminus\widetilde B)
			\ll \Big(\frac{\log n}n\Big)^{\frac{p(d+1)+2+2\beta}{d+1+2\beta}}.
		\end{equation}
		
		From \eqref{diff} we also obtain that, assuming $\tau_n$, 
		\begin{equation}%\label{Dnegyzet}
			D_if(Z^\ast)^2\ll \Big(\frac{\log n}n\Big)^{\frac{2(d+1)}{d+1+2\beta}},
		\end{equation}
		%$D_if(Z^\ast)^2\ll \Big(\frac{\log n}n\Big)^{\frac{2(d+1)}{d+1+2\beta}}$, 
		for all $i\in\{1,2,3\}$ and $Z^\ast\in\{Z,Z'\}$.
		We recall from \cite{TTW18}, that assuming $\tau_n$, it holds for $j\in\{2,3\}$ and $Y^\ast\in\{Y,Y'\}$, that
		$$\{D_{1,j}f(Y^\ast)\neq0\}\subset\{y^\ast_1\in B^d\setminus \widetilde B\}\cap\Big\{y^\ast_j\in \bigcup_{x\in\Vis_{y^\ast_1}(n)}\Vis_x(n)\Big\},$$
		and 
		\begin{equation}%\label{D1j}
			\Ex[\indi(D_{1,j}f(Y^\ast)\neq0)\indi(\tau_n)]\ll\PP_\beta(B^d\setminus\widetilde B)\sup_{z\in B^d\setminus\widetilde B}\PP_\beta\Big(\bigcup_{x\in\Vis_z(n)}\Vis_x(n)\Big).
		\end{equation}

		The previous union is contained in a cap whose diameter is of order $\big(\frac{\log n}{n}\big)^{\frac{1}{d+1+2\beta}}$, so its height is of order $\big(\frac{\log n}{n}\big)^{\frac{2}{d+1+2\beta}}$, and therefore its beta-content can be estimated as
		\begin{equation}%\label{Delta}
			\Delta(n):=\sup_{z\in B^d\setminus\widetilde B}\PP_\beta\Big(\bigcup_{x\in\Vis_z(n)}\Vis_x(n)\Big)\ll\frac{\log n}n.
		\end{equation}

		Analogously, assuming $\tau_n$, it holds for $y_1=y_1'$, that
		$$\{D_{1,2}f(Y)\neq0\}\cap\{D_{1,3}f(Y')\neq0\}\subset\{y_1\in B^d\setminus \widetilde B\}\cap\Big\{\{y_2,y_3'\}\in \!\!\bigcup_{x\in\Vis_{y_1}(n)}\!\!\Vis_x(n)\Big\},$$
		and, since the case $y_1\neq y_1'$ gives a smaller factor,
		\begin{equation}\label{D12D13}
			\Ex[\indi(D_{1,2}f(Y)\neq0)\indi(D_{1,3}f(Y')\neq0)\indi(\tau_n)]\ll\PP_\beta(B^d\setminus\widetilde B)\Delta(n)^2.
		\end{equation}

		Now, using \eqref{diff-ex}--\eqref{D12D13}, we obtain
		\begin{align*}
			Q_1&\ll \Big(\frac{\log n}n\Big)^{\frac{4(d+1)}{d+1+2\beta}} \Big(\frac{\log n}{n}\Big)^{\frac{2+2\beta}{d+1+2\beta}}
			\Big(\frac{\log n}{n}\Big)^{2}\ll\Big(\frac{\log n}{n}\Big)^{\frac{6(d+1)+2+6\beta}{d+1+2\beta}},    \\
			Q_2&\ll \Big(\frac{\log n}n\Big)^{\frac{4(d+1)}{d+1+2\beta}} \Big(\frac{\log n}{n}\Big)^{\frac{2+2\beta}{d+1+2\beta}}
			\frac{\log n}{n}\ll\Big(\frac{\log n}{n}\Big)^{\frac{5(d+1)+2+4\beta}{d+1+2\beta}},    \\
			Q_3&\ll \Big(\frac{\log n}n\Big)^{\frac{4(d+1)+2+2\beta}{d+1+2\beta}},    \\
			Q_4&\ll \Big(\frac{\log n}n\Big)^{\frac{3(d+1)+2+2\beta}{d+1+2\beta}}.
		\end{align*}
		
		The terms in \eqref{normal-bound} can be estimated from above as
		\begin{align*}
			\frac{\sqrt n}{\var V_s(K_n^\beta)}\sqrt{n^2 Q_1}&\ll n^{-\frac12+\frac{1+\beta}{d+1+2\beta}} (\log n)^{\frac{3(d+1)+1+3\beta}{d+1+2\beta}},    \\
			\frac{\sqrt n}{\var V_s(K_n^\beta)}\sqrt{n Q_2}&\ll n^{-\frac12+\frac{1+\beta}{d+1+2\beta}} (\log n)^{\frac{\frac52(d+1)+1+2\beta}{d+1+2\beta}},    \\
			\frac{\sqrt n}{\var V_s(K_n^\beta)}\sqrt{Q_3}&\ll n^{-\frac12+\frac{1+\beta}{d+1+2\beta}} (\log n)^{\frac{2(d+1)+1+\beta}{d+1+2\beta}},    \\
			\frac{n}{(\var V_s(K_n^\beta))^\frac32}Q_4 &\ll n^{-\frac12+\frac{1+\beta}{d+1+2\beta}} (\log n)^{\frac{3(d+1)+2+2\beta}{d+1+2\beta}}.
		\end{align*}
		
		Thus, for the Wasserstein distance, we obtain
		\begin{multline*}%\label{cht-becsles}
			d_W\Bigg(\frac{V_s(K_n^\beta)-\Ex V_s(K_n^\beta)}{\sqrt{\var V_s(K_n^\beta)}},G\Bigg)\ll n^{-\frac12+\frac{1+\beta}{d+1+2\beta}} \Big((\log n)^{\frac{3(d+1)+1+3\beta}{d+1+2\beta}} \\ + (\log n)^{\frac{\frac52(d+1)+1+2\beta}{d+1+2\beta}} 
			+ (\log n)^{\frac{2(d+1)+1+\beta}{d+1+2\beta}} + (\log n)^{\frac{3(d+1)+2+2\beta}{d+1+2\beta}}. \Big)
		\end{multline*}

		\section{Proof of Theorem~\ref{thm:f-vector}}
		First, we sketch the proof of the upper bound. 
		Our argument is similar to that of Reitzner \cites{R05, R05b}, and is based on a suitable modification of the proof of the upper bound of Theorem~\ref{thm:beta-main} in Section~\ref{sec:beta-upper}.
		
		The variance is again estimated from above using the Efron-Stein inequality.
		\begin{equation}\label{kfaces-ES}
			\Var f_k(K_n^\beta)\ll n \cdot\Ex[(f_k(K_{n+1}^\beta)-f_k(K_n^\beta))^2\indi(T_n)],
		\end{equation}
		where $T_n$ is the condition that the floating body $\widetilde B$ is contained in $K_n^\beta$.
		
		Let $F_n$ denote the number of facets of $K_n^\beta$ that can be seen from $x_{n+1}$. 
		If $x_{n+1}\in K_n^\beta$, then $F_n=0$ and $K_{n+1}^\beta$ has the same number of $k$-dimensional faces as $K_n^\beta$.
		Otherwise, $F_n$ is positive.
		For each $k\in\{0,1,\ldots,d-1\}$, let $f_k^+$ denote the number of $k$-faces of $K_{n+1}^\beta$ that are not contained in $K_n^\beta$, and let $f_k^-$ denote the number of $k$-faces of $K_n^\beta$ that are not faces of $K_{n+1}^\beta$.
		We recall from \cite{R05b}, that $|f_k(K_{n+1}^\beta)-f_k(K_n^\beta)|\leq f_k^++f_k^-$, furthermore,
		$$f_k^+\leq \binom{d}{k}F_n, \quad\text{ and }\quad f_k^-\leq\binom{d}{k+1}F_n,$$
		therefore, $\Ex[(f_k(K_{n+1}^\beta)-f_k(K_n^\beta))^2]\ll\Ex[F_n^2].$
		
		We recall some notation from Section~\ref{sec:beta-upper} with $s=d$ (in this case, the projections are the identity map, and therefore, not indicated):
		$I,J\subset\{1,\ldots,n\}$ are index sets with $d$ elements each, $F_I$ and $F_J$ denote the convex hulls of $x_{i_1},\ldots,x_{i_d}$, and $x_{j_1},\ldots,x_{j_d}$, respectively.
		Moreover, $U_A$ and $U_B$ denote  the events that $F_I$ and $F_J$, resp., is visible from $x_{n+1}$.
		Then \eqref{kfaces-ES} can be estimated from above as follows.
		\begin{align*}
			\eqref{kfaces-ES}&\ll n\cdot\Ex[F_n^2\cdot\indi(T_n)] \\
			&\ll n\cdot  \int_{(B^d)^{n+1}} \Biggl
			(\sum_{I} \indi(U_A)\Biggr )\Biggl (\sum_{J} \indi(U_B)\Biggr )  \indi(T_n)\prod_{i=1}^{n+1}f_{d,\beta}(x_i) \,\d X_{n+1}.
		\end{align*}
		
		The integral is very similar to \eqref{nagyonhosszu}, the only differences are that now $s=d$, the projections are removed, and the terms $\Vol_d([F_I, x_{n+1}])$ and $\Vol_d([F_J, x_{n+1}])$ are missing.
		A calculation analogous to the one in the proof of the upper bound in Theorem~\ref{main:upper-bound} yields an upper bound of order $n^{\frac{d-1}{d+1+2\beta}}$.
		%The proof of the upper bound in Theorem~\ref{main:upper-bound} can be modified to an analogous calculation, which yields an upper bound of order $n^{\frac{d-1}{d+1+2\beta}}$.
		%$$\eqref{kfaces-int}\ll n^{\frac{d-1}{d+1+2\beta}}.$$
		
		The proof of the lower bound of $\Var f_k$ is based on that of the lower bound of the variance of the intrinsic volumes in Section~\ref{sec:beta-lower}, and follows the argument of \cite{R05} as well, suitably modified for the beta-distribution.
		
		Let the caps $C(y_j,\gamma t_n)$ be defined as in Section~\ref{sec:beta-lower}, as well as the simplices $\Delta_i(y_j,t_n)$, for $j=1,\ldots,m$ and $i\in\{0,1,\ldots,d\}$.
		Let $B_j$ denote the event that $\Delta_0(y_j,t_n)$ contains two random points out of $x_1,\ldots,x_n$, each $\Delta_i(y_j,t_n)$, $i=1,\ldots,d$ contains one random point, and no other random points are contained in $C(y_j,\gamma t_n)$.
		For the probability of $B_j$, $j=1,\ldots,m$, by Lemmas~\ref{sapka-PC} and ~\ref{szimplex-beta}, we have
		\begin{equation*}
			\PP(B_j)\gg\binom{n}{d+2}\left(\frac{1}{n}\right)^{d+2}\left(1-\frac1n\right)^{n-d-2}\gg 1.
		\end{equation*}
		
		Assuming $B_j$, let $Y_j$ and $Z_j$ denote the points in $\Delta_0(y_j,t_n)$, and let $z_{j,1},\ldots,z_{j,d}$ denote the points in $\Delta_i(y_j,t_n)$, $i=1,\ldots,d$.
		For a fixed $j$, the convex hull of $\{Y_j,Z_j,z_{j,1},\ldots,z_{j,d}\}$ is either a simplex, or both $Y_j$ and $Z_j$ are vertices. 
		Both of them arise with positive probability.
		Thus, $f_k([Y_j,Z_j,z_{j,1},\ldots,z_{j,d}])$ can take at least two different values with positive probability.
		Thus, taking the variance for $Y_j$ and $Z_j$,
		\begin{equation*}
			\Var_{Y_j,Z_j} f_k([Y_j,Z_j,z_{j,1},\ldots,z_{j,d}])\gg 1.
		\end{equation*}

		Then, similarly to \cite{R05}, we obtain
		\begin{align*}
			\Var f_k(K_n^\beta)&\geq \Ex\Big(\sum_{j=1}^m \Var_{Y_j,Z_j} f_k([Y_j,Z_j,z_{j,1},\ldots,z_{j,d}]) \indi(B_j)\Big)\\
			&\gg m\cdot\PP(B_j)\gg n^{\frac{d-1}{d+1+2\beta}}.
		\end{align*}

		\section{Proof of Theorem~\ref{upper-bound2}}
		\subsection{The spherical case}
		%We need to show that if $K$ has a rolling ball and slides freely in a ball on the sphere, then its image $\overline{K}$ under the gnomonic projection has a rolling ball and slides freely in a ball in Euclidean space.
		
		%First, let us consider the rolling ball.
		Assume that a spherical ball of radius $r$ rolls freely in $K$.
		Let $e_{d+1}$ be the center of the circumsphere of $K$. Then $K$ is contained in the upper hemisphere. %Let $g$ be the gnomonic projection. %from $\Sp^d$ to $\R^d$ with tangent point $N$. 
		%Without loss of generality, we may assume that $N$ is the north pole of $\Sp^d$. 
		
		For a boundary point $x\in\bd K$, denote by $B_x$ the ball of radius $r$ for which $x\in\bd B_x$ and $B_x\subset K$. 
		Then $g$ maps balls in $\Sp^d_+$ to ellipsoids in $\R^d$, and those balls, whose center is $e_{d+1}$, are mapped to Euclidean balls in $\R^d$.
		Thus, for each $x\in\bd K$, $\overline{B}_x=g(B_x)$ is an ellipsoid in $\overline{K}$.
		Due to compactness, there is a maximal $\kappa_{\max}$ among the principal curvatures of all $\overline{B}_x$. 
		Then, a Euclidean ball of radius $1/\kappa_{\max}$ rolls freely in any $\overline{B}_x$ (see \cite{Sch14}*{Corollary 3.2.13.}) and, in turn, it rolls freely in $\overline{K}$, thus $\overline{K}$ also has a rolling ball.
		
		Now, assume that $K$ slides freely in a spherical ball of radius $R$.
		Then its spherical polar body $K^\ast$ has a rolling ball of radius $(\pi/2-R)$ 
		%(For information on spherical polarity, see, for example, Schneider \cite{Sch22}.).
		This can be seen in the following way:
		For a boundary point $x$ on $\bd K^\ast$, there is at least one supporting hypersphere $G^\ast$. Its polar $G$ is a point on the boundary of $K$.
		Since $K$ has a sliding ball, a ball $B_G$ of radius $R$ contains $K$, and $G$ is on its boundary.
		Then $B_G^\ast$ is a ball of radius $(\pi/2-R)$ contained in $K^\ast$ and intersects $G^\ast$ at a boundary point $y$ of $K^\ast$.
		If $y\neq x$ then $\bd K^\ast$ contains a great circle segment, thus $\bd K$ must have a point where it is not smooth.
		However, this contradicts the assumption that $K$ has a rolling ball, therefore $x=y$, and $K^\ast$ also has a rolling ball.
		
		Note that the assumptions on $K$ yield that
		the rolling ball inside $K^{\ast}$ is in the lower open hemisphere for all $x\in\bd K^{\ast}$.
		
		Let $\tilde{g}$ be the gnomonic projection
		whose center is $-e_{d+1}$. Let $(\cdot)^\circ$ denote Euclidean polarity in $\R^d$.  It is known, see Schneider \cite{Sch22}*{Lemma 3.2.2.}, that $g(K)^\circ=\tilde{g}(K^\ast)$. By the above argument, the existence of a rolling ball of $K^\ast$ implies that its image, $\tilde{g}(K)^\circ$ has a rolling ball too.
		Hug \cite{H00}*{Prop. 1.45.} proved that if a convex body $L\subset \R^d$ has a rolling ball, then $L^\circ$ slides freely in a ball of finite radius.
		Hence, $g(K)$ slides freely in a ball.
		
		The gnomonic projection maps $\Vol_{\Sp^d}$ into the Lebesgue measure in $\R^d$ with the density $\psi(x)=(1+\|x\|^2)^{-(d+1)/2}$ (see \cite{BW16}*{Proposition 4.2}), and $\Vol_{\Sp^d} (K_n)$ has the same distribution as the weighted volume of the convex hull of $n$ i.i.d. random points in $\R^d$ (cf. \cite{BT20}*{Section 5.1}).
		Thus, we may apply Theorem~\ref{thm:R-weighted-upper}, as the boundary conditions are satisfied for $\overline{K}$.
		
		\subsection{The hyperbolic case}
		Assume that the center of a maximal radius inscribed ball $B_1$ in $K$ is the point $e_{d+1}$. %$C=(0,\ldots,0,1)$.
		%Since this problem is translation-invariant, let us translate $K$ such that the center of an inscribed ball $B_1$ in $K$ is the point $C=(0,\ldots,0,1)$.
		Then $h(e_{d+1})=o$ and $h(B_1)=\overline{B}_1\subset\inti B^d$ in $\R^d$.
		
		Suppose that $K$ slides freely in a ball of radius $R$. 
		Let $B_2$ denote the smallest ball centered at $e_{d+1}$ that contains the union of all the sliding balls of radius $R$ of $K$. 
		The boundary of $K$ is contained in the spherical shell determined by the concentric balls $B_1$ and $B_2$, and $\bd \overline{K}$ is in the (Euclidean) spherical shell determined by $\overline{B}_1$ and $\overline{B}_2$.
		
		The gnomonic image of a hyperbolic ball contained in $B_2$ is an ellipsoid in $\overline{B}_2$ (contained in $\inti B^d$ in $\R^d$).
		It is clear that in the images of such balls
		the ratio of the ellipsoid's largest and shortest axes is bounded from above.
		
		For a fixed boundary point $x\in\bd K$, the image of the sliding ball $B_x$ of $K$ at $x$ is an ellipsoid that slides freely in a Euclidean ball of radius $R_x$.
		By compactness, there is a maximum radius $\Tilde{R}$ among all such $R_x$. Then $\overline{K}$ slides freely in a radius $\Tilde{R}$ ball.
		
		Now suppose that a ball of radius $r$ rolls freely in $K$.
		Then at every boundary point $x$, a ball inside of $K$ intersects $K$ at $x$. 
		The gnomonic projection of this ball is an ellipsoid in $\overline{K}$ for which the ratio of the axes is bounded.
		There is a maximal principal curvature $\kappa_{\max}$ among all these ellipsoids, and a ball of radius $1/\kappa_{\max}$ rolls freely in all of them. Consequently, it rolls freely in $\overline{K}$.
		
		The hyperbolic volume of $K$ has the same distribution as the weighted volume of the convex hull of $n$ i.i.d. random points chosen by the probability density function $\psi/\int_{\overline{K}}\psi(x)\,\d x$ in $\R^d$, where $\psi=(1-\|x\|^2)^{-(d+1)/2}$, $x\in\inti B^d$ is the density of the image measure of $\Vol_{\Hy^d}$.
		For more information, see \cite{BT20}*{Section 5.2}.
		
		Thus, the conditions of Theorem~\ref{thm:R-weighted-upper} are satisfied, so we may apply its statement, which yields the desired asymptotic variance upper bound for $\Vol_{\Hy^d}(K_n)$.

		\section{Direct proof of Theorem~\ref{upper-bound2}}\label{sec:direct}
		Essentially the same way as in Section~\ref{sec:beta-upper}, the variance upper bound in Theorem~\ref{upper-bound2} can be shown via a non-euclidean version of the economical cap theorem below.
		
		\begin{lemma}\label{gnom-nagysagrendek}
			The gnomonic projection preserves the order of magnitude of the volume of caps.
		\end{lemma}
		\begin{proof}
			Assume that $K\in\mathcal{K}(\Sp^d)$ and $e_{d+1}$ is the center of the minimum radius ball (circumball) containing $K$. Note that the radius $R_K$ of the circumball is less than $\pi/2$. 
			Consider a cap $C$ of $K$ with $\Vol_{\Sp^d} (C)=v$.     
			Then
			\begin{equation*}
				V(g(C))\le v\cdot \frac{1}{\langle u, e_{d+1}\rangle^{d+1}},
			\end{equation*}
			where $u$ is the farthest point of $C$ from the center of projection, $e_{d+1}$. On the one hand, $V(g(C))>v$.
			On the other hand, the quantity $\frac{1}{\langle u, e_{d+1}\rangle^{d+1}}$ is bounded from above by a constant $c_{K,d}$ depending only on $K$ and $d$.
			%By choosing the center of projection suitably, $c_{K,d}$ is determined by the circumradius of $K$. 
			Thus, $V(g(C))\leq c_{K,d}\cdot v$.
			%    On the other hand, $\Bar{t}_{\min}\ge t$. \todo{Mi a $\Bar{t}_{\max}$?}
			
			For the hyperbolic case, assume again that the center of the circumball of $K\in\mathcal{K}(\Hy^d)$ is $e_{d+1}$. 
			Let $D$ be a hyperbolic cap of $K$ with $\Vol_{\Hy^d} (D)=v$. Then $V(h(D))< v$,
			and $V(h(D))\geq v\cdot \frac{1}{\langle u, e_{d+1}\rangle^{d+1}}$, where $\frac{1}{\langle u, e_{d+1}\rangle^{d+1}}$ is bounded from below, similarly to the spherical case.  
		\end{proof}
		
		\begin{tetel}%\label{gazdasagos-sapkafedes}
			Let $K\in\mathcal{K}(\Sp^d)$ (or $\mathcal{K}(\Hy^d)$, resp.) and $0<v< (2d)^{-2d}\cos^{d+1}R_K$  (or $0<v< (2d)^{-2d}$ resp.). Then there exist caps $C_1,\dots,C_m$ and pairwise disjoint convex sets $C'_1,\dots,C'_{m'}$, with $m\approx m'$, such that
			\begin{enumerate}[(i)]
				\item $\bigcup_{i=1}^{m'}C'_i\subset K(v)\subset \bigcup_{i=1}^m C_i$,
				\item $\Vol_{\M^d} C_i\ll v$, $(i=1,\dots,m)$ and $\Vol_{\M^d} C'_i\gg v$, $(i=1,\dots,m')$, where $\M^d=\Sp^d$ or $\Hy^d$,  \label{ecct-nagysagrendek}
				\item for every cap $C$ of volume at most $v$, there is a $C_i$ containing $C$,
				\item every convex set $C'_i$ is contained in some $C_j$.
			\end{enumerate}
		\end{tetel}
		
		\begin{proof}
			We only give the proof for the case of $\Sp^d$, the hyperbolic variant can be shown essentially in the same way.
			The wet part $K(v)$ is the union of caps $C_\alpha$ of spherical volume at most $v$.
			%Let us map these caps to $\R^d$ by the gnomonic projection. 
			The images of the caps $C_\alpha$ under the gnomonic projection $g$ are the caps $\overline{C}_\alpha=g(C_\alpha)$ in $\overline{K}=g(K)$ with possibly different volumes of $\Bar{v}_\alpha$.
			By compactness, there exists a minimal and a maximal volume among $\Bar{v}_\alpha$, let us denote these by $\Bar{v}_{\min}$ and $\Bar{v}_{\max}$, respectively.
			Consider the (Euclidean) wet parts in $\overline{K}$ with parameters $\Bar{v}_{\min}$ and $\Bar{v}_{\max}$. Then
			\begin{equation*}
				\overline{K}(\Bar{v}_{\min})\subset g(K(v)) \subset \overline{K}(\Bar{v}_{\max}).
			\end{equation*}
			
			%  \textcolor{red}{$t$ mekkora kell, hogy legyen, hogy $\Bar{t}_{\max}<(2d)^{-2d}$ teljesüljön?}
			
			Apply Theorem~\ref{R-ecct} to both wet parts.
			Denote the caps and the convex sets contained in them provided by Theorem~\ref{R-ecct} corresponding to $\Bar{v}_{\max}$ by $D_i$ and $D'_i$ $(i=1,\dots,m)$ and those corresponding to $\Bar{v}_{\min}$ by $E_i$ and $E'_i$ $(i=1,\dots,m)$. Let $C_i=g^{-1}(D_i)$, $(i=1,\dots,m)$ and $C'_i=g^{-1}(E'_i)$ for some indices $i=1,\dots,m'$ specified later.
			
			The large caps $D_i$ cover $\overline{K}(\Bar{v}_{\max})$ so they cover $g(K(v))$ as well. 
			Furthermore, for any spherical cap $C$ of (spherical) volume at most $v$, its image $g(C)$ has volume at most $\Bar{v}_{\max}$, therefore, it is contained in some $D_i$.
			
			The sets $E'_i$ are disjoint and contained in caps of volume $\Bar{v}_{\min}$, therefore they are contained in $g(K(v))$, and there are caps $D_{j(i)}$ such that $E'_i\subset D_{j(i)}$ for each $i\in\{1,\dots,m\}$. 
			However, the same index $j$ can belong to multiple $i$'s, so some of the sets $E'_i$ can be dropped to avoid multiplicity. 
			The number of these sets cannot be too large due to the criteria on the volumes of $D_i$ and $E'_i$.
			The preimages of the remaining sets will be the spherical sets $C'_i$, $(i=1,\dots,m)$ with $m\approx m'$.
			
			Lemma~\ref{gnom-nagysagrendek} and the corresponding part of the Euclidean theorem yield \eqref{ecct-nagysagrendek}.
		\end{proof}

		\section{Acknowledgements}
		The second author was supported by the M\'oricz Doctoral Scholarship.
		
		The second author was also supported by the University Research Scholarship Programme (EKÖP) no. EKÖP-452-SZTE, which has been implemented with the support provided by the Ministry of Culture and Innovation of Hungary and the National Research, Development and Innovation Fund.
		
		This research was supported by NKFIH
		project no. 150151, which has been implemented with the support provided by
		the Ministry of Culture and Innovation of Hungary from the National
		Research, Development and Innovation Fund, financed under the
		ADVANCED\_24 funding scheme.
		
		This research was supported by project TKP2021-NVA-09. Project no. TKP2021-NVA-09 has been implemented with the support provided by the Ministry of Innovation and Technology of Hungary from the National Research, Development and Innovation Fund, financed under the TKP2021-NVA funding scheme.

	\end{document}